\documentclass{amsart} %texlab: ignore

\usepackage{
  bm,
  bbm,
  color,
  a4wide,
  amssymb,
  amsmath,
  latexsym,
  calligra,
  mathrsfs,
  stmaryrd,
  hyperref,
  graphicx,
  mathtools,
  enumitem,
  tikz,
  tikz-cd
}

\usepackage[2cell,arrow,curve,matrix]{xy}

\newcommand{\p}{\mathfrak{p}}
\newcommand{\q}{\mathfrak{q}}

\newcommand{\fa}{\mathfrak{a}}
\newcommand{\fb}{\mathfrak{b}}
\newcommand{\fm}{\mathfrak{m}}

\newcommand{\bo}{\mathbf{o}}

\newcommand{\cF}{\mathcal F}

\newcommand{\cN}{\mathcal N}

\newcommand{\cP}{\mathcal P}

\newcommand{\N}{\mathbb N}
\newcommand{\Z}{\mathbb Z}
\newcommand{\Q}{\mathbb Q}

\newcommand{\C}{\mathbb C}

\DeclareMathOperator*{\im}    {\mathop{\sf Im}\nolimits}

\renewcommand{\ker}          {\mathsf{Ker}}

\newcommand{\xto}[1]{\xrightarrow{#1}}

\newcommand{\Cl}{\mathrm{Cl}}

\newcommand{\qtext}[1]{\quad\text{#1}\quad}

\let\scong\cong
\renewcommand{\cong}{\;\scong\;}
\newcommand{\bigast}{\mathop{\Large\ast}}

\newtheorem{Core}{Definition}[section]
\newtheorem{Aux}{Remark}[section]

\newtheorem{De}[Core]{Definition}
\newtheorem{Th}[Core]{Theorem}
\newtheorem{Pro}[Core]{Proposition}
\newtheorem{Le}[Core]{Lemma}
\newtheorem{Cor}[Core]{Corollary}

\newtheorem{Note}[Aux]{Note}
\newtheorem{Rem}[Aux]{Remark}

\newtheorem{Ex}[Aux]{Example}

\author{Ilia Pirashvili}
\address{University of Galway, Áras De Brún, Gaillimh/Galway, H91 H3CY, Ireland}
\email{ilia\_p@ymail.com (personal)}
\email{ilia.pirashvili@universityofgalway.ie (work)}

\newcommand{\nEl}{\mathsf{nEll}}
\newcommand{\El}{\mathsf{Ell}_1}
\newcommand{\EG}[1][n]{_{\ensuremath{#1}}\mathsf{EG}}
\newcommand{\ER}[1][n]{_{\ensuremath{#1}}\mathsf{ER}}

\begin{document}
\title{\(n\)-ary elliptic groups, rings, and primes in arithmetic progressions}


\begin{abstract}
  I introduced the notion of an elliptic group in \cite{p11}. It is a quasi-group based
  on the tangent-chord law of elliptic curves and thus, becomes an abelian group upon
  singling out an element. This close proximity to abelian groups is reflected in the
  theory, and among other things, we can define elliptic rings, which are monoidal
  objects in elliptic groups. An other way of expressing this is to say that they are
  commutative monoids with an elliptic group structure that distributes over them.

  In this paper, we generalise this theory from the binary elliptic group structure to
  the \(n\)-ary structure, which we call \(n\)-ary elliptic groups and \(n\)-ary elliptic
  rings. The latter are once again (binary) commutative monoids with an \(n\)-ary
  operation that distributes over the monoidal structure in an \(n\)-ary sense.

  The key interest of these objects for us is their arithmetic properties, which are
  surprisingly pleasant. The key result is that Dirichlet's famous theorem on arithmetic
  progressions becomes simply Euclid's theorem in these \(n\)-ary rings, at least for
  progressions of the form \(an + 1\).

  Motivated by the hope to eventually prove this \(n\)-ary Euclidean theorem purely
  algebraically using the theory of \(n\)-ary rings (and thus give an alternative and
  purely algebraic proof of Dirichlet's theorem), we start by exploring the first
  arithmetic facts of these objects, including introducing the \(n\)-ary class group and
  showing that it indeed captures the unique \(n\)-ary factorisation. We also obtain a
  type of Dedekinds theorem for our main \(n\)-ary ring of interest: \(\nEl(\Z)\).
\end{abstract}

\maketitle

\section{Introduction}

  In \cite{p11}, I introduced the notion of an \emph{elliptic group}, which is a set
  with a binary operation derived by the geometric properties of elliptic curves
  (Bézout's theorem). It is a very well known and classical fact that the so called
  tangent-chord law defies a binary operation that satisfies the following properties:
  \begin{eqnarray*}
    x\ast y                 & = & y\ast x \\
    x\ast (x\ast y)         & = & y       \\
    x\ast (y\ast (z\ast w)) & = & w\ast (y\ast (z\ast x))
  \end{eqnarray*}
  The first one is our commutativity property, the second one (involution) is the
  closes we have to an inverse. Indeed, upon singling out a single point, this property
  is exactly the one that gives us inverses. The third and last property is a type of
  shuffle property that becomes the associativity law in the associated abelian group.
  These laws can be found, among many places, in \cite{silverman}.

  Elliptic groups are, as one would expect, very intimately linked to abelian groups.
  Indeed, as mentioned already, upon singling out any element \(o\), if we define the
  new operation as \(a +_o b := (a\ast b)\ast o\), we get the classical abelian group
  associated to the elliptic curve (which also happens to be its Picard group). The
  theory of elliptic groups is, however, slightly more general, as distinct elliptic
  groups can give rise to the same abelian groups, at least over \(\Q\), though it can
  not happen over \(\C\). This only happens when the obtained groups have \(\Z/3\Z\) as
  factors. Indeed, the number \(3\) appearing is no coincidence, as we will discuss a
  little further down in more detail.\\

  We can consider morphisms between these objects to respect the operation and thus
  obtain a category of elliptic groups. Here, we note that these morphisms will become
  group homomorphisms upon distinguishing an element.

  As with abelian groups, we can consider monoidal objects inside the category of
  elliptic groups and thus, we come to the notion of \emph{elliptic rings}. These are
  commutative monoids whose structure we denote by \(\circ\), over which our binary
  operation \(\ast\) distributes over.

  Starting with an abelian group \(G\), we can form an elliptic group \(\El(G)\) and
  likewise, starting with a commutative ring \(R\) with unit, we can form an elliptic
  ring \(\El(R)\). The operations are obtained by \(a\ast b:= 1 - a - b\) and \(a\circ
  b:= a + b - 3ab\) respectively.

  In \cite{p11}, we introduced this construction for \(\El(\Z)\) and, as a small
  example of this, we showed that Dirichlet's theorem for \(3n - 1\) is actually
  Euclid's theorem in this setting.

  Again, we observe that \(3\) seems to make special appearances, which is not a
  surprise, as elliptic group/ring theory is fundamentally based on elliptic curves,
  which are, of course, curves or degree \(3\). The idea was to see if we could
  generalise this theory to higher-ary structures and thus, obtain similar theory for
  other integers \(n\).\\[2pt]

  In this paper, we explore precisely this train of thought. We define the theory of
  \(n\)-ary rings as commutative monoids with a certain type of \(n\)-ary structure
  (which intakes \(n\) elements and outputs \(1\), and should be thought of as a type
  of addition) that distributes over it. Of course, for \(n=2\), this returns our
  original definition of elliptic groups, respectively elliptic rings.

  This harkens to the general idea of \emph{the geometry of monoids}, that stipulates
  that the multiplicative monoid of a commutative ring carries arguably more
  information than its additive groups and that one should perhaps think of rings not
  as abelian groups with a multiplicative structure, but as (commutative) monoids with
  an additive structure.%TO_DO cite Connes.

  We once again show that its possible to define \(n\)-ary rings based on classical
  rings, a construction we denote by \(\nEl(R)\), in particular based on \(\Z\), and
  show that inside these, the number \(n+1\) does indeed play a special role. One of
  our main theorems \ref{thm:main} says that Dirichlet's theorem for arithmetic
  progressions of the type \(an + 1\) is exactly Euclid's theorem for the \(n\)-ary
  ring \(\nEl(\Z)\). I believe that this observation warrants a further study of these
  objects.\\

  In more detail, we start by introducing and studying general \(n\)-ary (elliptic)
  groups in Section~\ref{sec:nary_ell_groups}. These are sets with an \(n\)-ary
  operation generalising the three axioms of elliptic groups and can be found in
  Definition~\ref{de:nary_group} explicitly. In the immediate example afterwards, we
  show that any group gives rise to such a structure (indeed, one for each element
  \(g\in G\)) in a fairly natural manner, a construction which we will denote by
  \(\nEl_g(G)\).\\

  In like manner, the next section \ref{sec:rings} will introduce the notion of
  \(n\)-ary (elliptic) rings which are effectively the monoid objects in the category
  of \(n\)-ary groups, though we will not formulate it in such a manner. Instead, we
  will simply define them as commutative monoids under an operation \(\circ\) with the
  \(n\)-ary group structure \(\ast\) distributing over it in the following way:
  \[
    (a_1 \ast_n \cdots \ast_n a_n) \circ_n b
    \;=\; (a_1 \circ_n b) \ast_n \cdots \ast_n (a_n \circ_n b),
  \]
  for all \(a_i, b\) in our \(n\)-ary ring \(R\). We will then proceed to develop some
  very basic terms, such as irreducibility, primness, ideals and congruences, but leave
  the general study for later.\\

  This is due to the fact that we will proceed to narrow our focus. Specifically, much
  as abelian groups give rise to \(n\)-ary groups, so do commutative rings give rise to
  \(n\)-ary rings (with the \(n\)-ary group structure being precisely the structure
  obtained form the additive group of the base ring). We introduce this construction in
  Section~\ref{sec:nary_rings_def} and denoted in like-manner to \(n\)-ary groups by
  \(\nEl(R)\). Note that here, we do not use the subscript \(g\), since we can no long
  pick any group element at our discretion due to the distributivity constraint. As
  such, we always assume \(g = 1\) in this paper. We note, however, that divisors of
  \(n+1\) all give rise to such structures in Paragraph~\ref{sec:role_of_r_nary_rings},
  though do not proceed to explore it more depth, yet.

  Having introduced this \emph{standard \(n\)-ary ring} \(\nEl(R)\), we will starts its
  exploration in Section~\ref{sec:nary_rings} and introduce the so called \emph{norm}
  function \(\Sigma_n: \nEl(R)\to R\). It will play a central role throughout the
  paper, as it allows us to reduce numerous questions to the base ring. In turn, its
  ``inverse'' \(J\) (defined by the pre-image of \(\Sigma_n\)) allows us to "lift"
  various constructions and tasks to the \(n\)-ary setting.

  This is a relatively larger section that explores some of the foundational properties
  needed for developing a basic algebraic study of \(\nEl(R)\), including properties of
  the norm function, the kernel (in the special case when \(n+1\) is invertible in
  \(R\), as we have not defined a kernel in general), and first isomorphism theorem,
  localisations and the norm-ideal correspondence, which will play an important role
  throughout the paper.\\

  Having developed the basic algebraic properties of \(\nEl(R)\), we will define and
  explore its basic arithmetic properties in Section~\ref{sec:arith_R}. This can be
  seen as both part II of the same section and part I of the following, and in many
  ways serves as a bridge.

  It deals with units and irreducibility, primeness and maximality, as well as touching
  on notions such as \(n\)-ary integral domains and \(n\)-ary fields.\\

  We will proceed to narrow our focus further and proceed to the first of our three
  main sections, namely Section~\ref{sec:arith_Z}. Here, we study the arithmetic
  properties of \(\nEl(\Z)\) in more detail. In particular, we will show that an
  element \(p\) is \(\circ_n\)-prime in \(\nEl(\Z)\) if and only if its norm
  \(|\Sigma_n|\) is (a classical) prime in \(\Z\).

  Moreover, we give an explicit criterion when \(\circ_n\) irreducibles and
  \(\circ_n\)-primes agree.\\

  This is particularly relevant since we focus on the main motivation of this theory in
  the following Section~\ref{sec:dirichlet}. This is summed up in
  Theorem~\ref{thm:main} where we effectively say that Dirichlet's theorem for a
  progression of the form \(1 + (n+1)\Z\) is effectively Euclid's theorem in the
  \(n\)-ary ring \(\nEl(\Z)\).

  Proving Euclid's theorem for \(\circ_n\)-irreducibles is fairly trivial and done in
  Theorem~\ref{thm:euclid_irred}. However, to obtian Dirichlet's theorem we would need
  to prove Euclid's theorem for \(\circ_n\)-primes, which we are not yet able to do
  without invoking Dirichlet.\\

  This indicates an interest in developing the algebraic number theory of \(\nEl(\Z)\)
  and \(n\)-ary rings in general. We take the very first steps in this direction in our
  penultimate section, namely Section~\ref{sec:class_group}, where we introduce the
  class group of \(\nEl(\Z)\) and show that it is always finite and trivial if an only
  if \(\nEl(\Z)\) has unique factorisaiton (or equivalently, is a PID). We show that
  this is the case if and only of \(n\in\{2,3,5\}\). Due to the finiteness of the
  classgroup, it remains an open question if it is still possible to sharpen our
  arguments to prove that there are infinitely many \(\circ_n\)-primes in \(\nEl(\Z)\)
  and thus, get a distinct and purely algebraic proof of Dirichlet's theorem for
  progressions of the form \(an + 1\).\\

  The last section \ref{sec:localisation} deals with the case \(\nEl(\Z_{n+1})\). As
  already hinted at when mentioning the kernel, things get significantly easier when
  \(n+1\) is invertible in \(\Z\) and we focus on that case in \(\nEl(\Z_{n+1})\). And
  indeed, the theory simplifies significantly. However, it does not bring us any closer
  to proving our desired theorem, but it does show arguebly more explicitly where the
  main obstruction lies.

\subsection*{Conventions}

    Throughout the paper, \(n\geq 2\) is a fixed intiger. Whenever we say ring, we
    assume commutative and with unity. We emphasise that \(0\) is the
    \emph{multiplicative} unit of the monoid \((R, \circ_n, 0)\) and should not be
    confused with the additive (classical \(+\)) unit, nor with the unit of \(\ast_n\)
    (indeed, \(\ast_n\) does not have a unit in general, the closest we get is
    \(\frac{1}{n+1}\) when it exists).

\section{\(n\)-ary elliptic groups}\label{sec:nary_ell_groups}

  The aim of this section is to introduce the notion of an \(n\)-ary (elliptic) group
  and give the main example of our interest in Example~\ref{ex:main}. We will also
  briefly talk about its lins to abelian groups.

\subsection{The axioms}

    \begin{De} \label{de:nary_group}
      An \emph{\(n\)-ary elliptic group} is a set \(S\) equipped with an \(n\)-ary
      operation \(\ast_n: S^n \to S\). For for all \(a_1, \ldots, a_n \in S\), it must
      satisfy the following three properties:
      \begin{enumerate}
        \item[(EG1)] For all permutations \(\sigma \in \mathfrak{S}_n\), we have
          \[
            a_{\sigma(1)} \ast_n \cdots \ast_n a_{\sigma(n)} \;=\; a_1 \ast_n \cdots \ast_n a_n.
          \]
        \item[(EG2)] Denote \(\hat{a} = a_1\ast_n \cdots \ast_n a_{n-1}\). For any \(s\in
          S\), we have
          \[ \left(s \ast_n \hat{a}\right) \ast_n \hat{a} \;=\; s. \]
          \emph{NOTE}: \(\hat{a}\) is merely a short-hand symbol and has no meaning!
          \(\ast_n\) is \emph{not} defined on \(n-1\) elements.
        \item[(EG3)] Let \(x_1,\ldots,x_{n-1},\, y_1,\ldots,y_{n-1},\, z,\,
          w_1,\ldots,w_{n-1} \in S\) and denote the short-hand notations
          \begin{eqnarray*}
            \hat{x} & = & x_1,\ldots,x_{n-1} \\
            \hat{y} & = & y_1,\ldots,y_{n-1} \\
            \hat{w} & = & w_1,\ldots,w_{n-1}.
          \end{eqnarray*}
          We have
          \[
            \tikz[baseline=(eq.base)]{
              \node (eq) {\(
              \hat{x} \ast_n \left( \hat{y} \ast_n \left( z\ast_n \hat{w}\right)\right)
              \)};

              \node (x) at ([yshift=-1pt]eq.south west) {};
              \node (w) at ([yshift=-1pt]eq.south east) {};

              \draw[<->]
              ([xshift=8pt,yshift=0pt]eq.south west)
              .. controls +(0,-0.35) and +(0,-0.35) ..
              ([xshift=-16pt,yshift=0pt]eq.south east);
            }
            \;=\;
            \hat{w}\ast_n \left( \hat{y} \ast_n \left( z\ast_n\hat{x}\right)\right).
          \]
      \end{enumerate}

      A morphism between \(n\)-ary groups is a structure preserving map. That is to say,
      \(f: S \to T\) is a map of \(n\)-ary groups if
      \[ f(a_1 \ast_n \cdots \ast_n a_n) \;=\; f(a_1) \ast_n \cdots \ast_n f(a_n). \]
      The category of \(n\)-ary elliptic groups is denoted by \(\EG\).
    \end{De}

    (EG1) says that \(\ast_n\) is \emph{totally symmetric}, and can be thought of as
    our version of commutativity.

    (EG2) will be reffered to as \emph{Involution} and is the closest we have to
    inverses. It allows us to express any single element in terms of other elements, by
    moving the remaining \(n-1\) elements on the other side.

    (EG3) is our version of associativity. In combination with (EG1), it allows us to
    move any element outside, or inside, a double bracket.

    \begin{Rem}
      The name \emph{\(n\)-ary elliptic group} is chosen out of continuity to the first
      paper \cite{p11}. While there might also be a geometric link, this is not explored
      and not the subject of this paper. Indeed, we will often refer to it as \(n\)-ary
      ring or \(n\)-ary group for short.
    \end{Rem}

    {\bf{Notation:}} Unless there is potential for confusion, we will simply write
    \(\ast\) instead of \(\ast_n\).

    \begin{Ex} \label{ex:main}
      Let \((G, +)\) be an abelian group. For any \(g\in G\), we can define an \(n\)-ary
      operation by setting
      \[ a_1 \ast \cdots \ast a_n \;:=\; g - \sum_{i=1}^n a_i. \]
      Straightforward calculation shows that this defines an \(n\)-ary elliptic group. To
      spell them out: (EG1) holds due to the symmetric nature of our definition.

      \noindent For (EG2), we denote \(\hat{a}:= a_2\ast_n \cdot \ast_n a_n =
      \sum\limits_{i=2}^n a_i\) and see that
      \begin{eqnarray*}
        \left(a_1 \ast_n \hat{a}\right)\ast_n \hat{a} & = & \left(g - \hat{a} - a_1  \right) \ast_n \hat{a} \\
                                                      & = & g - \hat{a} - \left(g - \hat{a} - a_1  \right)  \\
                                                      & = & a_1
      \end{eqnarray*}
      (EG3) is proved like-wise, and we will use notation in like-manner. We expand
      \begin{eqnarray*}
        \hat{x} \ast_n \left( \hat{y} \ast_n \left( z\ast_n \hat{w}\right)\right) & = & \hat{x} \ast_n \left( \hat{y} \ast_n \left( g - z - \hat{w} \right) \right) \\
                                                                                  & = & \hat{x} \ast_n \left( g - \left( g - z - \hat{w} \right) -  \hat{y} \right) \\
                                                                                  & = & g - z  - \hat{w} +  \hat{y} - \hat{x}
      \end{eqnarray*}
      and observe that its symmetric with respect to \(\hat{x}\) and \(\hat{w}\), giving
      us our proof. These are equal as desired. We call this the \emph{standard \(n\)-ary
      elliptic group over} \(g\) and denote it by \(\nEl_g(G)\).
    \end{Ex}

    {\bfseries Notation:} Instead of writing \(\nEl_1(G)\), we will simply write
    \(\nEl(G)\).

    \begin{Rem}[Notation warning]
      To avoid any undue confusion, we note that we used the notation \(\El(\Z)\) in
      \cite{p11}. Though the elliptic group as defined in said paper is using \(g = 1\)
      in the \(\ast\) operation, the subscript \(1\) in \(\El(\Z)\) corresponds to
      something else. In the notation of this paper, we have \(\El(\Z) = 2\mathsf{El}(\Z)
      (= 2\mathsf{El}_1(\Z)\) if one wants to be very precese).
    \end{Rem}

\subsection{Descending maps to abelian groups} \label{sec:descend_to_groups}

    Let \((S, \ast_n\) be an \(n\)-ary elliptic group and fix any element \(o\in S\).
    Define the \(n-1\)-ary operation
    \[
      a_1 \ast_{n-1} \cdots \ast_{n-1} a_{n-1}
      \;:=\; a_1 \ast_n \cdots \ast_n a_{n-1} \ast_n o,
    \]
    for \(a_i\in S\). This gives rise to a functor
    \[ \ast_n^{o}:\; \EG \to \EG[n-1]. \]
    Upon iterating this, we can descend all the way to the binomial setting. Denote the
    sequence of elements used by \(\bo = (o_1, \ldots, o_{n-2}) \in S^{n-2}\). We thus
    have the composite functor
    \[
      \ast^\bo
      \;:=\; \ast_3^{o_{n-2}} \circ \cdots \circ \ast_n^{o_1}:\; \EG \to \EG[2].
    \]
    Thus, we can use our results in \cite{p11}. Among other things, we have:

    \begin{Le}
      For any \(\bo \in S^{n-2}\) and \(o\in S\), the operation
      \[ a +_\bo b \;:=\; o \ast^{\bo} (a \ast^{\bo} b), \]
      defines abelian group group structure on \(S\).
    \end{Le}

\section{\(n\)-ary elliptic rings}\label{sec:rings}

  \begin{De}
    An \emph{\(n\)-ary elliptic ring} is a (commutative) monoid \((A, \circ_n, e)\),
    endowed with an \(n\)-ary group structure \(\ast_n: A^n \to A\) such that for all
    \(a_1, \ldots, a_n, b \in A\), the \(n\)-ary distributivity
    \begin{equation} \label{eq:dist}
      (a_1 \ast_n \cdots \ast_n a_n) \circ_n b
      = (a_1 \circ_n b) \ast_n \cdots \ast_n (a_n \circ_n b)
    \end{equation}
    holds. For simplicity, we will generally write \(A\) instead of \((A, \circ_n, e,
    \ast_n)\).

    A \emph{morphism} of \(n\)-ary elliptic rings \(f: A \to B\) is a monoid homomorphism
    that is simultaneously a morphism of \(n\)-ary elliptic groups.

    The category of \(n\)-ary rings is denoted by \(\ER\).
  \end{De}

  {\bfseries{Notation:}} Unless there is potential for confusion, we will simply write
  \(\circ\) instead of \(\circ_n\).

  \begin{Rem}
    It is clear that by symmetry (commutativity of \(\circ\)), the distributivity law
    \eqref{eq:dist} holds from the other side as well, meaning
    \[
      b \circ (a_1 \ast \cdots \ast a_n)
      \;=\; (b \circ a_1) \ast \cdots \ast (b \circ a_n).
    \]
  \end{Rem}

  The descend from \(n\)-ary elliptic groups to \(n-1\)-ary elliptic groups discussed
  in Section~\ref{sec:descend_to_groups} translates to a descend for \(n\)-ary elliptic
  rings, that is to say, upon fixing an element \(o\in A\), we have a functor
  \[ \ast_n^o: \ER \to \ER[n-1]. \]
  In like-manner to Section~\ref{sec:descend_to_groups}, we can thus obtain a functor
  to (binary) elliptic rings introduced in \cite{p11}.

\subsection{\(n\)-ary elliptic ideals}

    Our aim is to explore the theory of \(n\)-ary algebraic number theory, and, as
    expected, ideal-theory play a crucial role. However, much as (ring) ideals are not
    simply monoid ideals, but monoid ideals that respect the additive structure,
    \(n\)-ary ideals must also be closed under the second operation. We make this
    concrete with the following definition:

    \begin{De}
      A subset \(I\subseteq A\) of an \(n\)-ary elliptic ring \(A\) is called an
      \emph{\(n\)-ary ideal} if
      \begin{enumerate}
        \item[(I1)] for all \(a_1, \ldots, a_n \in I\),
          we have \(a_1 \ast_n \cdots \ast_n a_n \in I\) and
        \item[(I2)] for all \(a\in A\), \(r\in I\), we have \(a \circ_n r \in I\).
      \end{enumerate}
    \end{De}

    If the context is clear, we will simply say ideal. In particular, an \(n\)-ary
    ideal is a monoid ideal of the underlying \(\circ_n\)-monoid. As the theory of
    monoid schemes shows, this now allows us to effectively do algebraic geometry in
    this setting.

    \begin{Rem}
      We note that, just like for monoids, \(n\)-ary elliptic ideals may be empty.
    \end{Rem}

    \begin{De}
      A set of the form
      \[ (a) \;:=\; \{ a \circ_n r \mid r \in A \} \]
      for an element \(a\in A\) is called a \emph{principal ideal}.
    \end{De}

    The fact that it is an ideal immediately follows form
    \[
      a_1 \ast_n \cdots \ast_n a_n \;=\; (p \circ_n r_1) \ast_n \cdots \ast_n (p \circ_n r_n)
      \;=\; p \circ_n (r_1 \ast_n \cdots \ast_n r_n) \in (p). \qedhere
    \]

\subsection{Congruences and quotients}

    A congruence on an \(n\)-ary elliptic ring \(A\) is, as one would expect from the
    general theory of universal algebra, an equivalence relation \(\sim\subseteq A^2\)
    such that \(A/\sim\) is an \(n\)-ary elliptic ring in the natural way and the
    canonical surjection \(A\to A/\sim\) is an \(n\)-ary ring homomorphism. That is to
    say, if \(\sim\) \emph{respects} the operations \(\circ_n\) and \(\ast_n\).

\subsubsection{Congruences form ideals}

      For a non-empty \(n\)-ary ideal \(I\subseteq A\) of an \(n\)-ary elliptic ring
      \(A\), we can define a congruence by setting \(x\sim_I y\), \(x, y\in A\) if and
      only if there exist \(a_1,\ldots, a_{n-1}, b_1,\ldots, b_{n-1}\in I\), such that
      \begin{eqnarray} \label{eq:congruence}
        x \ast a_1 \ast \cdots \ast a_{n-1} = y \ast b_1 \ast \cdots \ast b_{n-1}.
      \end{eqnarray}
      We call the elements \(a_1,\ldots, a_{n-1}\), \(b_1,\ldots, b_{n-1}\)
      \emph{witnesses} of the congruence.

      If \(I\) is empty, we simply define \(\sim_I\) to be trivial, meaning \(x \sim_I
      y\) if and only if \(x = y\).

      \begin{Rem}
        Compare this with the way we defined congruences on monoids (by assuming \(\ast\)
        were out monoid operation). Note that if \(\ast\) were the standard addition on a
        commutative ring, we could simplify \(x + a = y + b\) to \(x - y = b - a \in I\),
        since \(+\) has inverses. In the monoidal setting, if \(+\) did not have inverses
        (e.g. in \(\N\), the congruence would simply be defined as \(x + a = y + b\)
        without further simplification. Our definition is thus simply the most obvious
        choice.
      \end{Rem}

      \begin{Rem} \label{rem:simpler_eq}
        While we can not directly subtract, we can use (EG2) to do something a bit
        similar. Specifically, we can \(\ast_n\) the elements \(a_1,\ldots, a_{n-1}\) on
        both sides to rephrase Condition~\eqref{eq:congruence} as
        \[
          x \;=\; (y \ast b_1 \ast \cdots \ast b_{n-1})\ast a_1 \ast \cdots \ast a_{n-1}
        \]
        This form is often more convenient for proofs, as the appropriate witnesses can
        be derived rather than produced from thin air.
      \end{Rem}

      \begin{Pro}
        The relation \(\sim_I\) is a congruence relation on \((A, \ast, \circ)\).
      \end{Pro}

      \begin{proof}
        By construction, there is nothing to prove if \(I\) is empty. Hence, we will
        assume throughout that its non-empty.

        We start by checking that its an equivalence relation. Reflexivity is clear and
        so is symmetry. Transitivity takes a little work, but is not very hard.
        Specifically, let \(x \sim_I y\) and \(y \sim_I z\) be witnessed by
        \[
          x\ast \hat{a} \;=\; y\ast \hat{b}, \qquad y\ast \hat{c} \;=\; z\ast \hat{d},
        \]
        where \(\hat{a} = a_1 \ast \cdots \ast a_{n-1}\), and similarly for \(\hat{b},
        \hat{c}\), and \(\hat{d}\), all with members in \(I\). Note that as in
        Definition~\ref{de:nary_group} and in other places, \(\hat{a}\) is merely a
        short-hand notation, and can not be evaluated.

        We use (EG2) on the second equation to expose \(y\) as
        \[ y \;=\; (z \ast \hat{d}) \ast \hat{c} \]
        and substitute it into the first equation. We get
        \[ x \ast \hat{a} \;=\; ((z \ast \hat{d}) \ast \hat{c}) \ast \hat{b}. \]
        Recalling \(\hat{d} = d_1 \ast \cdots \ast d_{n-1}\), write \(\bar{\hat{d}}:=d_1
        \ast \cdots \ast d_{n-2}\) to express \(\hat{d} = \bar{\hat{d}}\ast d_{n-1}\). We
        have singled out \(d_{n-1}\) but of course, we could have singled out any other
        terms as well. Apply (EG3) to RHS to move \(z\) outside
        \begin{eqnarray*}
          ((z \ast \hat{d}) \ast \hat{c}) \ast \hat{b}                       & = & 
          ((z \ast \bar{\hat{d}}\ast d_{n-1}) \ast \hat{c}) \ast \hat{b}  \\
                                                                             & = & 
          ((\hat{b} \ast d_{n-1}) \ast \hat{c}) \ast z \ast \bar{\hat{d}} \\
                                                                             & = & 
          z \ast  ((\hat{b} \ast d_{n-1}) \ast \hat{c}) \ast \bar{\hat{d}},
        \end{eqnarray*}
        where we used \((EG1)\) at the end. Summarising our manipulation gives
        \[
          x \ast \hat{a}
          \;=\; z \ast ((\hat{b} \ast d_{n-1}) \ast \hat{c}) \ast \bar{\hat{d}},
        \]
        Setting \(e_1 = (\hat{b} \ast d_{n-1}) \ast \hat{c}\) and \(e_i = d_{i-1}\) for
        \(2 \leq i \leq n-1\), we have
        \[ x\ast \hat{a} \;=\; z\ast e_1\ast \cdots\ast e_{n-1}. \]
        It remains to check that \(e_1, \ldots, e_{n-1} \in I\). For \(e_2 = d_1, \ldots,
        e_{n-1} = d_{n-2}\) this is trivial. For \(e_1\), we observe that every component
        of \(\bar{b}\), \(\bar{c}\) and \(d_{n-1}\) are elements of \(I\) and by (I1),
        \(e_1\in I\) follows.\\

        We now proceed to prove that it is a congruence relation. We will show this for
        \(n=3\) since the rest is just ``bookkeeping''. Let \(x\sim_I x'\), \(y\sim_I
        y'\), \(z\sim_I z'\). That is to say, there exist \(a,a', b,b', c,c', d,d', e,e',
        f,f'\), such that
        \begin{eqnarray}
          x\ast a\ast b & = & x'\ast a'\ast b' \\
          y\ast c\ast d & = & y'\ast c'\ast d' \\
          z\ast e\ast f & = & z'\ast e'\ast f'
        \end{eqnarray}
        Use (- \(\ast a\ast b\)) (and likewise for \(c,d\) and \(e,f\)) on LHS and RHS
        from the right and EG2 to get
        \begin{eqnarray}
          x & = & (x'\ast a'\ast b')\ast a\ast b \label{eq:x} \\
          y & = & (y'\ast c'\ast d')\ast c\ast d \label{eq:y} \\
          z & = & (z'\ast e'\ast f')\ast e\ast f \label{eq:z}
        \end{eqnarray}
        We have to show that there exist \(g,g', h,h'\) such that:
        \begin{itemize}
          \item[\(\circ\)] \((x\circ y)\ast g\ast h = (x'\circ y')\ast g'\ast h'\). For
            this, observe
            \[
              \begin{array}{rclcl}
                x\circ y                                  & \;=\; & \big((x'\ast a'\ast b')\ast a\ast b\big) \circ
                \big((y'\ast c'\ast d')\ast c\ast d \big)
                                                          &       & Eq\ (\ref{eq:x}), (\ref{eq:y}) \\
                                                          & \;=\; & [\big((x'\ast a'\ast b')\ast a\ast b\big) \circ (y'\ast c'\ast d')]
                                                          &       & \,\text{Distributivity}\,      \\
                                                          &       & [\big((x'\ast a'\ast b')\ast a\ast b\big) \circ c]\ast
                                                          &       &                                \\
                                                          &       & [\big((x'\ast a'\ast b')\ast a\ast b\big) \circ d]\ast
                                                          &       &                                \\
                                                          & \;=\; & \cdots
                                                          &       &                                \\
              \end{array}
            \]
            We note that it will take the following form once we are done:
            \[
            x\circ y \;=\; (x'\circ y')\ast \,\text{ ``Terms with }\, u \circ v \,\text{ where at least one of }\, u,v \,\text{ is in }\, I \,\text{ ''.}\, \]
            Hence \(u\circ v\in I\). Using Remark \ref{rem:simpler_eq}, we are now done.
            \item[\(\ast\)] We have to show \((x\ast y\ast z)\ast g\ast h = (x'\ast y'\ast
            z')\ast g'\ast h'\). We take
            \[
              \begin{array}{rclcl}
                x \ast y\ast z & \;=\; & ((x'\ast a'\ast b')\ast a\ast b)\ast y \ast z
                               &       & Eq\ (\ref{eq:x}) \\
                               & \;=\; & ((x'\ast \underline{y\ast z})\ast a\ast b)\ast \underline{a' \ast b'}
                               &       & (EG3)            \\
                               & \;=\; & ((x'\ast ((y'\ast c'\ast d')\ast c\ast d)\ast z)\ast a\ast b)\ast a' \ast b'
                               &       & Eq\ (\ref{eq:y}) \\
                               & \;=\; & ((((y'\ast c'\ast d')\ast c\ast d)\ast \underline{x'}\ast z)\ast a\ast b)\ast a' \ast b'
                               &       & (EG1)            \\
                               & \;=\; & ((((y'\ast \underline{x'\ast z})\ast c\ast d)\ast \underline{c'\ast d'}))\ast a\ast b)\ast a' \ast b'
                               &       & (EG3)            \\
                               & \;=\; & ((((y'\ast x'\ast ((z'\ast e'\ast f')\ast e\ast f))\ast c\ast d)\ast c'\ast d')\ast a\ast b)\ast a' \ast b'
                               &       & Eq\ (\ref{eq:z}) \\
                               & \;=\; & ((((((z'\ast e'\ast f')\ast e\ast f)\ast \underline{y'\ast x'})\ast c\ast d)\ast c'\ast d')\ast a\ast b)\ast a' \ast b'
                               &       & (EG1)            \\
                               & \;=\; & ((((((z'\ast \underline{y'\ast x'})\ast e\ast f)\ast \underline{e'\ast f'}))\ast c\ast d)\ast c'\ast d')\ast a\ast b)\ast a' \ast b'
                               &       & (EG3).
              \end{array}
            \]
            The "sum" \(z'\ast y'\ast x'\) is now a single element and it is clear that
            through systematic use of (EG1) and (EG3), we can ``wriggle'' it all the way
            outside the brackets, which will leave us with something of the form \(x\ast
            y\ast z = (z'\ast y'\ast x')\ast g'' \ast h''\), and Remark \ref{rem:simpler_eq}
            will now give the desired result.
        \end{itemize}
      \end{proof}

      \begin{De}
        Define the \emph{quotient \(n\)-ary elliptic ring} \(A/I\) to be the set of
        equivalence classes
        \[ [a] \;=\; \{b \in A \mid b \sim_I a\} \]
        under the congruence \(\sim_I\), endowed with the operations
        \[
          [a_1] \ast_n \cdots \ast_n [a_n] \;:=\; [a_1 \ast_n \cdots \ast_n a_n],
          \quad
          [a] \circ_n [b] \;:=\; [a \circ_n b],
          \quad
          e_{A/I} \;:=\; [e].
        \]
        This makes the canonical map \(\pi: A \to A/I\), \(\pi(a) = [a]\) a surjection of
        \(n\)-ary elliptic rings.
      \end{De}

\subsection{Units, irreducibles and cancellativeness}

    We already mentioned that since an \(n\)-ary ring is a monoid (seen as
    multiplication) with an additional structure allows us to do the core constructions
    of algebraic geometry in this setting. In particular, we inherit notions of units,
    associate elements, divisibility, irreducibility, localisations and numerous
    similar things. The notions listed in this section are merely the notions of the
    multiplicative monoid of \(n\)-ary rings, which we merely list for the convenience
    of the reader. Remark~\ref{rem:cancellative} is still recommended to be looked at
    by all readers, to avoid confusion.

    \begin{De}[Units]
      A \emph{\(\circ_n\)-unit} is an element \(\in A\) if there exists \(v \in A\) such
      that \(u \circ_n v = e\), where \(e\) is the identity of \(\circ_n\). The set
      \(A^\times\) of all \(\circ_n\)-units forms an abelian group. Two elements \(a, b
      \in A\) are called \emph{\(\circ_n\)-associates}, written \(a\sim b\) if \(a = b
      \circ_n u\) for some unit \(u\). In other words, if they define the same principal
      ideals, meaning \((a) = (b)\). We say \(a\) \emph{\(\circ_n\)-divides} \(b\),
      written \(a \mid_n b\), if \(b \in (a)\).
    \end{De}

    \begin{De}[Cancellative]
      An element \(m\in M\) of a monoid \(M\) is \emph{cancellative} if for all \(a,b\in
      M\), \(ma = mb\) implies \(a = b\). An element \(z\) is called an \emph{absorbing
      element} (often called \emph{zero} as well) if for all \(m\in M\), \(zm = z\).

      A monoid \(M\) is cancellative if every element is cancellative. A monoid with a
      \emph{distinguished} absorbing element \(z\in M\) is called \emph{cancellative} if
      every element \(m\neq z\) is cancellative.
    \end{De}

    \begin{Rem}
      Note that when we say ``a monoid with a distinguished absorbing element'', we mean
      that the absorbing element is fixed and respected, rather than ``accidentally''. In
      such situations, we would generally require homomorphisms to respect them, ideals
      would be required to contain them (non-empty) etc. To simplify terminology, monoids
      with distinguished absorbing elements are called \emph{binoids} in \cite{p7}.
    \end{Rem}

    \begin{Rem}[\(n\)-Ary rings with absorbing elements] \label{rem:cancellative}
      For the \(\circ_n\)-monoid of an \(n\)-ary ring, we will show in
      Lemma~\ref{le:zero}(i) that there is a natural candidate for the absorbing element
      \(z\), namely \(z = 1/(n+1)\) (when it exists). In such cases, it is best to
      respect that element, hence our slightly convoluted definition of cancellativeness.
    \end{Rem}

    The term absorbing element was preferred over zero in this paper since \(0\) is
    already going to be overburdened in this paper (in a ring \(R\) it will be the
    additive unit, multiplicative zero/absorbing element and \(\circ_n\)-unit in
    \(\nEl(R)\), rather than a \(\circ_n\)-zero/\(\circ_n\)-absorbing element).

    \begin{De} \label{def:cancellative}
      An \(n\)-ary elliptic ring \(A\) is \emph{\(\circ_n\)-cancellative} if for all \(a,
      b, c \in A\),
      \[
        a \circ_n b \;=\; a \circ_n c \implies b \;=\; c \ \,\text{ or }\, \ a \;=\; z_A,
      \]
      where \(z_A\) denotes the absorbing element of \(A\) (if it exists).
    \end{De}

    Divisibility \(\mid_n\) defines a preorder. In a cancellative monoid, \(a \mid_n
    b\) and \(b \mid_n a\) implies \(a \sim b\). Moreover, if the monoid is sharp, that
    is to say, has only the single invertible element \(e\), the elements \(a\) and
    \(b\) are equal.

    \begin{De}
      A non-unit \(p \in A\) is \emph{\(\circ_n\)-irreducible} (or an \emph{atom}) if \(p
      = a \circ_n b\) implies \(a \in A^\times\) or \(b \in A^\times\).
    \end{De}

    \begin{Pro}[Factorisation in irreducibles for ACCP] \label{pro:factorisation_irred}
      An \(n\)-ary ring \(A\) satisfies ACCP (Ascending Chain Condition on Principal
      Ideals), meaning, every element can be written as a product of irreducibles.
    \end{Pro}

\subsection{Prime and maximal ideals}

    \begin{De}
      An ideal \(\p \subsetneq A\) is \emph{\(\circ_n\)-prime} if for all \(a, b \in A\),
      \[ a \circ_n b \in \p \implies a \in \p \,\text{ or }\, b \in \p. \]
      A non-unit element \(p \in A\) is \emph{\(\circ_n\)-prime} if \((p)\) is a prime
      ideal.
    \end{De}

    Note that the empty set is prime by default.

    \begin{Le}
      An ideal \(\p \subsetneq A\) is \(\circ_n\)-prime if and only if the complement \(A
      \setminus \p\) is closed under \(\circ_n\).
    \end{Le}

    This lemma basically states that we can consider localisations, and move towards
    algebraic geometry by considering \(n\)-ary ring schemes, among other
    continuations.

    \begin{De}
      An ideal \(\fm \subsetneq A\) is \emph{maximal} if there is no ideal \(I\) with
      \(\fm \subsetneq I \subsetneq A\).
    \end{De}

    \begin{Le}
      Let \(p \in A\) be \(\circ_n\)-irreducible. If \((p) \subseteq (a)\) for some \(a
      \notin A^\times\), then \((a) = (p)\).
    \end{Le}

    \begin{proof}
      By assumption \(p = a \circ_n r\) and \(a \notin A^\times\). The irreducibility of
      \(p\) now implies \(r \in A^\times\) and so, \((a) = (p)\).
    \end{proof}

    \begin{Pro}
      In an \(\circ_n\)-cancellative \(n\)-ary elliptic ring, every \(\circ_n\)-prime
      element is \(\circ_n\)-irreducible.
    \end{Pro}

    \begin{proof}
      If \(p = a \circ_n b\) and \(p \mid_n a\) (say \(a = p \circ_n c\)), we have \(p =
      p \circ_n (b \circ_n c)\). Since our monoid is cancellative, we can deduce \(e = b
      \circ_n c\). Subsequently \(b \in A^{\times}\).
    \end{proof}

\section{The \(n\)-ary elliptic rings \(\nEl(R)\)} \label{sec:nary_rings}

\subsection{Definition and first results} \label{sec:nary_rings_def}

    We introduce the following fundamental example of an \(n\)-ary ring.

    \begin{Pro}
      Let \(R\) be a ring. Define the following operations on \(R\):
      \[
        a_1 \ast \cdots \ast a_n \;:=\; 1 - \sum_{i=1}^n a_i,
        \qquad
        a \circ b \;:=\; a + b - (n+1)ab,
      \]
      and distinguish the element element \(e = 0 \in R\) as the \(\circ_n\)-unit. Then
      \((R, \ast, \circ, 0)\) is an \(n\)-ary elliptic ring, which we denote by
      \(\nEl(R)\).
    \end{Pro}

    \begin{proof}
      We already know that \((\nEl(R), \ast_n)\) an \(n\)-ary group by Example
      \ref{ex:main}. It remains to show that \((\nEl(R), \circ_n, 0)\) is a monoid and
      that \(\ast_n\) distributes over \(\circ_n\).

      That \((\nEl(R), \circ_n, 0)\) is a monoid is pretty clear, as commutativity and
      unit axioms are immediate and for associativity we note that
      \[
        (a \circ b) \circ c \;=\; a + b + c - (n+1)ab - (n+1)ac - (n+1)bc + (n+1)^2 abc,
      \]
      is symmetric in \(a, b, c\).

      For distributivity, let \(S = \sum_{i=1}^n a_i\). Then
      \begin{align*}
        (1 - S) \circ b & = 1 - S + b - (n+1)(1-S)b = 1 - S - nb + (n+1)Sb,
      \end{align*}
      and
      \begin{align*}
        \bigast_{i=1}^n (a_i \circ b) & = 1 - \sum_{i=1}^n (a_i + b - (n+1)a_ib)
        = 1 - S - nb + (n+1)Sb.
      \end{align*}
      These are equal and so, our distributivity condition \eqref{eq:dist} holds.
    \end{proof}

    By defining \(\nEl(f): \nEl(R) \to \nEl(S)\) in the obvious (pointwise) way from an
    \(n\)-ary ring homomorphism \(f: R \to S\), we note that
    \[ R \mapsto \nEl(R) \]
    is functorial.

\subsection{On the role of \(1\) in the definition of \(\nEl(R)\)}

    In the definition of a standard \(n\)-ary group, we could choose any element \(g\in
    G\) and obtain an \(n\)-ary group \(\nEl_g(G)\). It is thus a natural question to
    ask if \(1\) in our definition of \(n\)-ary rings is needed. As one might expect,
    \(1\) being the multiplicative unit is not a mere coincidence in this case, but it
    is not forced either, as we will show in our discussion here. However, it is a
    sensible choice, at least until we move to consider not a single \(n\)-ary ring
    structure but a full poset. This is, however, left for a future paper.

    Let \(R\) be a ring and \(r \in R\) any element. Set
    \[ a_1 \ast \cdots \ast a_n \;:=\; r - \sum_{i=1}^n a_i. \]
    This is an \(n\)-ary Group by Example~\ref{ex:main}. However, as soon as one wishes
    to equip \(R\) with a compatible \(\circ_n\)-multiplication and thereby obtain an
    \(n\)-ary elliptic \emph{ring}, the choice of \(r\) becomes constrained.

\subsubsection{The role of \(r\) in the \(n\)-ary ring case}
      \label{sec:role_of_r_nary_rings}

      The key constraint for us to define a binary operation \(a\circ b\) of the
      general form \(a + b - \lambda ab,\ \lambda\in R\) comes form \(\ast\) having to
      distribute over \(\circ\).

      A direct computation gives
      \begin{align*}
        \text{LHS of \eqref{eq:dist}} & = \Bigl(r - \sum_{i=1}^n a_i\Bigr)\circ b = r - S + b - \lambda\bigl(r-S\bigr)b  = r - S + b\bigl(1-\lambda r + \lambda S\bigr), \\[4pt]
        \text{RHS of \eqref{eq:dist}} & = \bigast_{i=1}^n(a_i\circ b)        = r - \sum_{i=1}^n\bigl(a_i + b - \lambda a_i b\bigr)        = r - S - nb + \lambda Sb,
      \end{align*}
      where \(S = \sum a_i\). Comparing the coefficients of \(b\) (the terms not
      involving \(S\)) gives \(1 - \lambda r = -n\), that is to say
      \[ \lambda r \;=\; n+1. \]
      It follows that for any unit \(r\in R^{\times}\), we may set \(\lambda =
      (n+1)r^{-1}\). However, the resulting \(n\)-ary elliptic ring
      \((R,\ast_r,\circ_\lambda,0)\) is going to be isomorphic to the standard
      (\(\lambda = n+1\)) one, as \(a\mapsto ra\) will give us an \(n\)-ary ring
      isomorphism.

      \medskip
      For non-invertible \(\lambda\), however, things are not so simple. We always have
      a \emph{dual} \(n\)-ary ring, given by
      \[ r \;=\; n+1 \,\qtext{and}\, \lambda \;=\; 1. \]
      This is in some sense the ``least'' \(n\)-ary ring with our default one
      (\(\lambda = n+1\)) being the ``greatest''. The other can be defined as the
      divisors of \(n+1\). This will give a lattice structure by divisibility and can
      be seen as a sort of ``sheaf'' in the \(\mathbb{F}_1\) setting (for this, we
      mention without going in detail that a sheaf in a category \(\mathbb{C}\) on an
      affine monoid scheme, is a functor form a lattice to \(\mathbb{C}\).

      \medskip\noindent
      This direction is not yet explored in our current work and a task for future
      papers.

\subsection{The norm map}

    One of our most important tools for studying \(\nEl(R)\) is the following monoid
    homomorphism, as it allows us to descend into the underlying ring \(R\).

    \begin{De}[Norm]
      We define the \emph{norm} map (or simply norm) \(\Sigma_n: \nEl(R) \to R\) by
      \[ \Sigma_n(a) \;:=\; 1 - (n+1)a. \]
    \end{De}

    \begin{Le} \label{lem:mult}
      For any ring \(R\), the norm map \(\Sigma_n: \nEl(R) \to R\) is a monoid
      homomorphism. That is, for all \(a, b \in R\),
      \[
        \Sigma_n(a \circ_n b)
        \;=\; \Sigma_n(a) \cdot \Sigma_n(b), \qquad \Sigma_n(0) \;=\; 1.
      \]
    \end{Le}

    \begin{proof}
      Direct computation gives us
      \begin{align*}
        \Sigma_n(a \circ_n b) & = 1 - (n+1)(a + b - (n+1)ab)                   \\
                              & = 1 - (n+1)a - (n+1)b + (n+1)^2 ab             \\
                              & = \bigl(1 - (n+1)a\bigr)\bigl(1 - (n+1)b\bigr) \\
                              & = \Sigma_n(a) \cdot \Sigma_n(b).
      \end{align*}
      By definition, \(\Sigma_n(0) = 1\).
    \end{proof}

    \begin{Le}[Sum identity] \label{lem:sum}
      For all \(a \in R\),
      \[ \Sigma_n(a) + \Sigma_n(1 - a) \;=\; 1 - n. \]
    \end{Le}

    \begin{proof}
      \(\Sigma_n(a) + \Sigma_n(1-a) = (1-(n+1)a) + (1 - (n+1)(1-a)) = 2 - (n+1) = 1-n\).
    \end{proof}

\subsection{On the invertibility of \(n+1\)}

    \begin{Le} \label{lem:sigma_image}
      If \(n+1 \neq 0\) in \(R\), then \(\Sigma_n: R \to R\) is injective, with image
      \(\im(\Sigma_n) = 1 + (n+1)R\).
    \end{Le}

    \begin{proof}
      Let \(\Sigma_n(a) = \Sigma_n(b)\). Then \((n+1)(a-b) = 0\) and so, \(a = b\) since
      \(n+1\neq 0\) in \(R\) and injectivity follows.

      The image satisfies \(1 - (n+1)a \equiv 1 \pmod{(n+1)R}\), so \(\im(\Sigma_n)
      \subseteq 1 + (n+1)R\). Conversely, given \(m \in 1+(n+1)R\), by definition \(m =
      1-(n+1)a\) for some \(a\in R\), and then \(\Sigma_n(a) = m\). Hence \(\im(\Sigma_n)
      = 1+(n+1)R\).
    \end{proof}

    On the opposite side of Lemma~\ref{lem:sigma_image}, we have the following trivial
    observation.

    \begin{Pro}
      Let \(n+1 = 0\) in \(R\). Then \(\circ_n\) agrees with the addition of \(R\).
      Meaning
      \[ (R, \circ_n, 0) \simeq (R, +, 0) \]
      is a monoid isomorphism and in particular, every element is a \(\circ_n\)-unit.
    \end{Pro}

    As such, while we regard \(\circ_n\) as the ``multiplication" and \(\ast_n\) as the
    ``addition'', in this case, things take a bit of a turn.

    From the arithmetic point of view, this case does not seem that interesting, and
    subsequently, we will focus on the case when \(n+1\) is not \(0\).

    We also describe the case when \(n+1\in R^\times\) is invertible, which is much
    more interesting to us.

    \begin{Le} \label{le:zero}
      Let \(n+1\in R^\times\) be invertible and denote \(z = 1 / (n+1)\). For every
      \(a\in R\), we have
      \begin{itemize}
        \item[(i)] \(a\circ_n z = z\);
        \item[(ii)] \(z\ast \cdots \ast z = z\);
        \item[(iii)] \(\Sigma_n(z) = 0\);
      \end{itemize}
    \end{Le}

    \begin{proof}
      (i) \(a\circ_n z = a + z - (n+1)az = a + \frac{1}{n+1} - a = z\).

      (ii) \( z\ast \cdots \ast z = 1 - n\cdot\frac{1}{n+1} = \frac{(n+1) - n}{n+1} =
      \frac{1}{n+1} = z.\)

      (iii) \(\Sigma_n(z) = 1 - (n+1)/(n+1) = 1 - 1 = 0\).
    \end{proof}

    \begin{Rem} \label{rem:n+1_invertible}
      In esssense, \(1/(n+1)\), if it exists, plays the role of the absorbing element in
      our \(n\)-ary ring. Note that these will be respected by morphisms, however, an
      ideal need not contain it. It would be better to require ideals to contain this
      element (if it exists), but is not vital. It is a special subcategory and should be
      treated as such. This discussion can be found in Section~\ref{sec:localisation}.
    \end{Rem}

\subsection{Product and localisation}

    Due to the functorial and pointwise nature of the construction \(\nEl(R)\), as well
    as the fact that localisation is purely monoidal and \(\circ_n\) is a classical
    commutative monoid, the following propositions are immediate and clear:

    \begin{Pro}
      For rings \(R\) and \(S\), there is a natural isomorphism of \(n\)-ary elliptic
      rings
      \[ \nEl(R \times S) \cong \nEl(R) \times \nEl(S), \]
      where the operations \(\ast_n\) and \(\circ_n\) in the RHS are defined
      componentwise.
    \end{Pro}

    \begin{Pro}
      Let \(R\) be a ring and \(T \subseteq R\) a multiplicatively closed subset. Then
      \[ \nEl(T^{-1}R) \cong T^{-1}\nEl(R), \]
      where \(T^{-1}\nEl(R)\) is the localisation of \(\nEl(R)\) with \(T\).
    \end{Pro}

\subsection{The norm-ideal correspondence when coming from a ring}
    \label{sec:ideals_nEl}

    While the ideals of \(\nEl(R)\) are not simply reducible to ideals of \(R\), they
    are nonetheless strongly controlled by the norm map \(\Sigma_n\), as it offers a
    systematic way to pass between classical ideals of \(R\) and \(n\)-ary elliptic
    ideals of \(\nEl(R)\).

    For example, it characterises principal ideals completely:

    \begin{Le}[Norm divisibility] \label{lem:norm_div}
      Let \(R\) be a ring. For \(p, a \in R\), we have
      \[
        (a)\subseteq (p) \,\text{ in }\, \nEl(R)
        \;\iff\; (\Sigma_n(a))\subseteq (\Sigma_n(p)) \,\text{ in }\, R.
      \]
      Moreover, in that case \(\Sigma_n(a) = \Sigma_n(p) \cdot k\) with \(k \in
      \im(\Sigma_n) = 1 + (n+1)R\), with the explicit representative being \(r = a -
      pk\).
    \end{Le}

    \begin{proof}
      Let \((a) \subseteq (p)\) be an inclusion of \(n\)-ary ideals, meaning \(a = p
      \circ_n r\). Lemma \ref{lem:mult} then gives the result, as \(\Sigma_n(a) =
      \Sigma_n(p)\cdot\Sigma_n(r)\) must hold so we have found \(k =
      \Sigma_n(r)\in\im(\Sigma_n)\).

      Assume conversely \((\Sigma_n(a))\subseteq (\Sigma_n(p))\), meaning \(\Sigma_n(a) =
      \Sigma_n(p)\cdot k\) and set \(r = a - pk\). Since
      \[
        \begin{array}{ccccc}
          \Sigma_n(a) &       & \;=\;               &       & \Sigma_n(p)\cdot k \\[4pt]
          \|          &       &                     &       & \|                 \\[2pt]
          1 - (n+1)a  & \;=\; & (1 - (n+1)p)\cdot k & \;=\; & k - (n+1)pk,
        \end{array}
      \]
      we get \(1 - (n+1)(a - pk) = k\). In other words, \(\Sigma_n(r) = k \in
      \im(\Sigma_n)\). We have
      \begin{align*}
        p \circ_n (a-pk) & = a + p(1-k) - (n+1)p(a-pk) \\
                         & = a + p\bigl[(1-k) - (n+1)(a-pk)\bigr].
      \end{align*}
      From \(\Sigma_n(a-pk) = k\), we have \(1-k = (n+1)(a-pk)\). The bracket vanishes
      and leaving us \(p \circ_n (a-pk) = p \circ_n r = a\) and we are done.
    \end{proof}

    \begin{Rem}
      The automatic membership \(k \in \im(\Sigma_n)\) is an important feature in
      \(n\)-ary elliptic rings. The "\(\circ_n\)-quotient" \(\Sigma_n(a)/\Sigma_n(p)\),
      whenever it exists in \(R\), is always a norm. For \(R = \Z\), this says that \(a
      \in (p)\) if and only if \(\Sigma_n(p) \mid \Sigma_n(a)\) in \(\Z\).
    \end{Rem}

    \begin{De}
      For an ideal \(\fa \subseteq R\), define
      \[
        J(\fa) \;:=\; \Sigma_n^{-1}(\fa)
        \;=\; \{\,x \in R \mid 1-(n+1)x \in \fa\,\}.
      \]
    \end{De}

    \begin{Pro} \label{prop:J_ideal}
      The set \(J(\fa)\) is an \(n\)-ary elliptic ideal of \(\nEl(R)\), for any ideal
      \(\fa \subseteq R\).
    \end{Pro}

    \begin{proof}
      (I1) Let \(x_1, \ldots, x_n \in J(\fa)\), so \(\Sigma_n(x_i) = \alpha_i \in \fa\),
      meaning \(1 - (n+1)x_i = \alpha_i\). Then \((n+1)x_i = 1 - \alpha_i\) and so
      \begin{eqnarray*}
        \Sigma_n(x_1 \ast \cdots \ast x_n)
                                           & = & \Sigma_n\!\Bigl(1 - \sum x_i\Bigr) \\
                                           & = & -n + (n+1)\sum x_i                 \\
                                           & = & -n + \sum(1-\alpha_i) = -\sum \alpha_i \in \fa.
      \end{eqnarray*}

      (I2) Let \(x \in J(\fa)\) and \(r \in R\). By Lemma \ref{lem:mult}, we have
      \(\Sigma_n(x \circ_n r) = \Sigma_n(x)\cdot\Sigma_n(r) \in \fa\).
    \end{proof}

    The construction \(J\) satisfies the following properties:

    \begin{Pro} \label{prop:J_ideal_obs}
      Let \(\fa\subseteq R\) be an ideal.
      \begin{itemize}
        \item[(i)] \(J(\fa) = \nEl(R)\) if and only if \(\fa = R\).
        \item[(ii)] \(J((0)) = \{x: (n+1)x = 1\}\), which equals \(\{1/(n+1)\}\) if
          \(n+1 \in R^{\times}\), and empty otherwise.
        \item[(iii)] The map \(\fa \mapsto J(\fa)\) is order-preserving.
        \item[(iv)] \(J(\fa) \neq \emptyset\) if and only if \(\fa\) contains an element
          of \(1+(n+1)R\).
        \item[(v)] For \(R = \Z\), \(J((p)) = \emptyset\) for all primes \(p \mid (n+1)\).
          That is to say, the primes dividing \(n+1\) are "invisible" to the
          \(\circ_n\)-structure.
      \end{itemize}
    \end{Pro}

    \begin{proof}
      (i) That \(J(R) = R = \nEl(R)\) is trivial. Conversely, let \(J(\fa) = \nEl(R)\).
      For every \(x \in R\), we thus have \(1 - (n+1)x \in \fa\). By taking \(x = 0\) we
      get \(1 \in \fa\) and so \(\fa = R\).

      (ii) \(J((0)) = \{x \in R \mid \Sigma_n(x) = 0\} = \{x \mid (n+1)x = 1\}\). If
      \(n+1 \in R^{\times}\), this has the unique solution \(x = 1/(n+1)\), otherwise
      there is no solution in \(R\).

      (iii) Let \(\fa \subseteq \fb\) and \(x \in J(\fa)\). This gives \(\Sigma_n(x) \in
      \fa \subseteq \fb\) and we deduce \(x \in J(\fb)\) and hence, \(J(\fa) \subseteq
      J(\fb)\).

      (iv) Saying \(J(\fa) \neq \emptyset\) is equivalent to saying that there exists an
      element \(x \in R\) with \(\Sigma_n(x) \in \fa\). This in turn is equivalent to
      \(\im(\Sigma_n) \cap \fa \neq \emptyset\). Since \(\im(\Sigma_n) = 1+(n+1)R\), this
      is exactly the condition that \(\fa\) contains an element of \(1+(n+1)R\).

      (v) Let \(R = \Z\) and \(x\in J((p))\). Then \(p \mid 1-(n+1)x\). Subsequently, if
      \(p \mid (n+1)\) is to hold as well, we get \(p = \pm 1\), contradicting that \(p\)
      is a prime.
    \end{proof}

    \begin{Pro} \label{prop:J_prime}
      If \(\p \subsetneq R\) is a prime ideal of \(R\), then \(J(\p)\) is an
      \(\circ_n\)-prime ideal of \(\nEl(R)\).
    \end{Pro}

    \begin{proof}
      Since \(1 \notin \p\), \(\Sigma_n(0) = 1 \notin \p\), so \(0 \notin J(\p)\). Hence
      \(J(\p) \neq \nEl(R)\). Assume \(x \circ_n y \in J(\p)\). Then \(\Sigma_n(x) \cdot
      \Sigma_n(y) = \Sigma_n(x \circ_n y) \in \p\) and without loss of generality,
      \(\Sigma_n(x) \in \p\), which gives us \(x \in J(\p)\).
    \end{proof}

\section{Arithmetic of \(\nEl(R)\)}\label{sec:arith_R}

  When our \(n\)-ary elliptic ring is of the form \(A = \nEl(R)\), we can control the
  arithmetic through the norm map \(\Sigma_n\). This section makes that control
  precise.

\subsection{Units, associates, and divisibility}

    \begin{Le} \label{le:units_general}
      An element \(u \in \nEl(R)\) is an \(\circ_n\)-unit if and only if \(\Sigma_n(u)
      \in R^{\times}\).
    \end{Le}

    \begin{proof}
      Let \(u\) be an \(\circ_n\), meaning there exists a \(v\in \nEl(R)\), such that \(u
      \circ_n v = 0\). It follows \(\Sigma_n(u)\cdot\Sigma_n(v) = 1\) by
      Lemma~\ref{lem:mult} and we obtain \(\Sigma_n(u) \in R^{\times}\).

      Conversly, set \(v = -u\cdot\Sigma_n(u)^{-1}\). Then
      \begin{align*}
        u \circ_n v
                    & = u - u\,\Sigma_n(u)^{-1} + (n+1)u^2\Sigma_n(u)^{-1} \\
                    & = u + u\,\Sigma_n(u)^{-1}\bigl((n+1)u - 1\bigr)      \\
                    & = u + u\,\Sigma_n(u)^{-1}\bigl(-\Sigma_n(u)\bigr)    \\
                    & = u - u = 0,
      \end{align*}
      where we used \((n+1)u - 1 = -\Sigma_n(u)\) in the third step. For uniqueness,
      observe that any inverse \(v'\) satisfies \(\Sigma_n(v') = \Sigma_n(u)^{-1}\), and
      the formula \(v = -u\Sigma_n(u)^{-1}\) is manifestly unique.
    \end{proof}

    \begin{Cor}
      \(\Sigma_n\) restricts to a group homomorphism \(\nEl(R)^{\times} \to R^{\times}\).
    \end{Cor}

    \begin{proof}
      Multiplicativity follows from Lemma~\ref{lem:mult}. For the last claim, observe
      that
      \[
        \Sigma_n(-u\Sigma_n(u)^{-1}) \;=\; 1 + (n+1)u\Sigma_n(u)^{-1} \;=\; (\Sigma_n(u) +
        (n+1)u)/\Sigma_n(u) \;=\; 1/\Sigma_n(u),
      \]
      where we used \(\Sigma_n(u) + (n+1)u = 1\).
    \end{proof}

\subsection{\(n\)-ary cancellativity}

    \begin{Le} \label{lem:cancel_sigma}
      For \(a, b, c \in \nEl(R)\),
      \[ a \circ_n b \;=\; a \circ_n c \;\iff\; (b - c)\,\Sigma_n(a) \;=\; 0. \]
    \end{Le}

    \begin{proof}
      Expanding both sides using \(x \circ_n y = x + y - (n+1)xy\), gives us
      \[
        \begin{array}{crcl}
                   & a + b - (n+1)ab       & \;=\; & a + c - (n+1)ac \\
          \;\iff\; & (b - c) - (n+1)a(b-c) & \;=\; & 0               \\
          \;\iff\; & (b-c)(1-(n+1)a)       & \;=\; & 0,
        \end{array}
      \]
      and subsequently, \(1 - (n+1)a = \Sigma_n(a)\).
    \end{proof}

    \begin{Pro} \label{prop:cancel_char}
      An \(n\)-ary ring \(A = \nEl(R)\) is \(\circ_n\)-cancellative if and only if no
      element of \(\im(\Sigma_n) \setminus \{0\}\) is a zero-divisor in \(R\).
    \end{Pro}

    \begin{proof}
      We know by Lemma~\ref{lem:cancel_sigma} that \(a \circ_n b = a \circ_n c\) if and
      only if \((b-c)\Sigma_n(a) = 0\). Cancellation fails if and only if there exist
      \(b\neq c\), and \(a\) with \(\Sigma_n(a)\neq 0\) such that \((b-c)\Sigma_n(a)=0\).
      But this is equivalent to saying that there exists an element of
      \(\im(\Sigma_n)\setminus\{0\}\) which is a zero divisor.
    \end{proof}

\subsection{\(n\)-ary fields}

\subsubsection*{Definition and characterisation}

      We will use the term \emph{group-like binoid}, which is a term introduced in
      \cite{p7}, being effectively a group with an absorbing element, which is
      respected by morphisms. To be more precise, a \emph{binoid} is a set \(M\) with a
      binary associative operation \(\cdot\) and a distinguished element \(z\in M\)
      such that \(z\cdot m = m\cdot z = z\) for all \(m\in M\). Morphisms between
      binoids \(M\) and \(N\) are monoid homomorphisms that additionally map \(z_M\) to
      \(z_N\). If for all elements \(m\in M\) with \(m\neq z\) we have an inverse, we
      will say that the binoid is \emph{group-like}.

      \begin{De} \label{def:nary_field}
        An \(n\)-ary elliptic ring \(A\) is an \emph{\(n\)-ary elliptic field} if the
        \(\circ_n\)-monoid \((A, \circ_n, 0_A)\) is a group or a group-like binoid.
        Equivalently, if \((A\setminus\{z_A\},\circ_n, 0_A)\) is a group.
      \end{De}

\subsection{Maximal implies prime, and the hierarchy}

    The aim of this section is to show that the classical implications field
    \(\Rightarrow\) integral domain \(\Rightarrow\) ring, and maximal \(\Rightarrow\)
    prime \(\Rightarrow\) ideal, carry over to the \(n\)-ary setting.

    \begin{Th} \label{th:field_is_cancellative}
      Every \(n\)-ary elliptic field is \(\circ_n\)-cancellative.
    \end{Th}

    \begin{proof}
      Let \(A\) be an \(n\)-ary elliptic field and suppose \(a\circ_n b = a\circ_n c\)
      with \(a\neq z_A\). By Definition~\ref{def:nary_field}, \(a\) has an
      \(\circ_n\)-inverse \(a'\) with \(a'\circ_n a = 0_A\). By Lemma~\ref{lem:mult},
      \(\Sigma_n\) is a monoid homomorphism and on \(A\setminus\{z\}\), it is an
      isomorphism onto \(R^{\times}\), since \(a\neq z\) if and only if \(\Sigma_n(a)\neq
      0\), and every nonzero element is a unit. In particular \(\circ_n\) is associative
      on \(A\setminus\{z\}\), as it is isomorphic to the associative group
      \((R^{\times},\cdot)\). If \(b = z\) or \(c = z\) then
      \[ a\circ_n b \;=\; a\circ_n z \;=\; z \;=\; a\circ_n c \]
      forces both to be \(z\), so \(b=c\). Otherwise, we have
      \[
        b \;=\; 0_A\circ_n b \;=\; (a'\circ_n a)\circ_n b \;=\; a'\circ_n(a\circ_n b)
        \;=\; a'\circ_n(a\circ_n c) \;=\; (a'\circ_n a)\circ_n c \;=\; 0_A\circ_n c
        \;=\; c. \qedhere
      \]
    \end{proof}

    \begin{Rem}
      For a general \(n\)-ary elliptic ring \(A\) (not necessarily of the form
      \(\nEl(R)\)), the implication maximal \(\Rightarrow\) prime, or equivalently,
      ideal-field \(\Rightarrow\) ideal-cancellative, remains open.
    \end{Rem}

\subsection{Irreducible and prime elements}

    The following propositions hold for any ring \(R\).

    \begin{Pro} \label{prop:prime_general}
      An element \(p \in \nEl(R)\) is \(\circ_n\)-prime if and only if \(\Sigma_n(p)\) is
      prime in the submonoid \((1+(n+1)R, \cdot, 1)\). That is to say, for all \(s, t \in
      1 + (n+1)R\),
      \[
        \Sigma_n(p) \mid st \,\text{ in }\, R
        \implies
        \Sigma_n(p) \mid s \,\text{ or }\, \Sigma_n(p) \mid t \,\text{ in }\, R.
      \]
    \end{Pro}

    \begin{proof}
      By Lemma \ref{lem:norm_div}, \(p \mid_n a\) if and only if \(\Sigma_n(p) \mid
      \Sigma_n(a)\). Since \(\Sigma_n(a\circ_n b) = \Sigma_n(a)\Sigma_n(b)\) and both
      factors lie in \(1 + (n+1)R\), the prime condition translates exactly to the stated
      condition.
    \end{proof}

    \begin{Cor} \label{cor:prime_suff}
      If \(\Sigma_n(p)\) is a prime element of \(R\), then \(p\) is \(\circ_n\)-prime in
      \(\nEl(R)\).
    \end{Cor}

    The converse for irreducibles fails in general. Let \(R=\Z\) and \(n\) be such that
    \((\Z/(n+1)\Z)^\times \not\cong \{\pm 1\}\). There exist \(\circ_n\)-irreducibles
    with composite norm, as can be seen inthe following example:

    \begin{Ex} \label{ex:main_counter_example}
      For \(n = 4\), \(-1\) is \(\circ_4\)-irreducible, yet \(\Sigma_4(-1) = 6 = 2\cdot
      3\), is a composite.
    \end{Ex}

    \begin{Pro}
      An element \(p \in \nEl(R)\) is \(\circ_n\)-irreducible if and only if
      \(\Sigma_n(p)\) is irreducible in the submonoid \((1+(n+1)R, \cdot, 1)\). That is
      to say, if it cannot be written as \(\Sigma_n(p) = s\cdot t\) with \(s, t \in 1 +
      (n+1)R\) and \(s, t \notin R^{\times}\).
    \end{Pro}

    \begin{proof}
      By Lemma \ref{lem:norm_div}, \(p = a\circ_n b\) if and only if \(\Sigma_n(p) =
      \Sigma_n(a)\Sigma_n(b)\), and \(a \in \nEl(R)^{\times}\) if and only if
      \(\Sigma_n(a) \in R^{\times}\) by Theorem \ref{le:units_general}.
    \end{proof}

    \begin{Le}[Coprimality lemma] \label{lem:coprime}
      For any \(r\in R\) in a UFD \(R\), we have \(\gcd(\Sigma_n(r), n+1) = 1\).
    \end{Le}

    \begin{proof}
      Any prime factor of both \(\Sigma_n(p)\) and \(n+1\) would divide \(1\).
    \end{proof}

    \begin{Pro}
      Let \(R\) be a UFD and \(n+1 \notin R^{\times}\). For a non-zero non-unit \(p \in
      \nEl(R)\), the following are equivalent:
      \begin{enumerate}
        \item[\rm(i)] \(p\) is \(\circ_n\)-prime.
        \item[\rm(ii)] \(\Sigma_n(p)\) is irreducible in \(R\).
      \end{enumerate}
    \end{Pro}

    \begin{proof}
      (i)\(\Rightarrow\)(ii):\enspace We need to show that \(p\) is prime in \(R\), the
      fact that \(R\) is a UDF will imply irreducibility. Suppose \(\Sigma_n(p)\mid ab\)
      in \(R\). Since \(R\) is a UFD, we can factorise elements in primes. We groups
      these by ones which are \(1\pmod{n+1}\) and these which are not \(1\pmod{n+1}\).
      Since the product of these primes that are \(1\pmod{n+1}\) are again of like-from,
      we can write \(a = \Sigma_n(a')u\) and \(b = \Sigma_n(b')v\) where \(u,v\) are
      products of prime factors of \(R\) not lying in \(1 + (n+1)R\).

      By the coprimality lemma~\ref{lem:coprime}, \(\Sigma_n(p) \mid \Sigma_n(a')
      \Sigma_n(b')\). By Proposition~\ref{prop:prime_general}, \(p\) we know that
      \(\Sigma_n(p)\) is prime in the submonoid \((1 + (n+1)R, \cdot, 1) \subseteq R\)
      and thus, without loss of generality, \(\Sigma_n(p)\mid\Sigma_n(a')\), giving
      \(\Sigma_n(p)\mid a\). As such, \(\Sigma_n(p)\) is a prime as \(R\) is a UFD, it is
      irreducible as well.

      (ii)\(\Rightarrow\)(i):\enspace This is just Corollary~\ref{cor:prime_suff}.
    \end{proof}

\subsection{Idempotents}

    \begin{Le}
      The \(\circ_n\)-idempotents of \(\nEl(R)\) are exactly those \(a \in R\) that
      satisfy \(a \cdot \Sigma_n(a) = 0\) in \(R\). In particular:
      \begin{itemize}
        \item[(i)] \(0\) is always an idempotent;
        \item[(ii)] If \(n+1 \in R^{\times}\), then \(1/(n+1)\) is idempotent;
        \item[(iii)] If \(R\) is an integral domain, the only \(\circ_n\)-idempotents are
          \(0\) and \(1/(n+1)\) (if \(n+1 \in R^{\times}\)).
      \end{itemize}
    \end{Le}

    \begin{proof}
      Being an \(\circ_n\)-idempotent mean \(a \circ_n a = a\), which expands to \(a -
      (n+1)a^2 = 0\). We get \(a\cdot\Sigma_n(a) = 0\). In an integral domain, \(a = 0\)
      or \(\Sigma_n(a) = 0\). The latter requires \(a = 1/(n+1)\) if \(n+1 \in
      R^{\times}\).
    \end{proof}

\section{Arithmetic of \(\nEl(\Z)\)}\label{sec:arith_Z}

  We now specialise to \(\nEl(\Z)\). Here, the general theory takes a particularly
  clean and explicit form.

\subsection{Units}

    \begin{Le} \label{lem:units}
      The only \(\circ_n\)-unit in \(\nEl(\Z)\) is \(0\). Equivalently, \(|\Sigma_n(a)| =
      1\) if and only if \(a = 0\) and, in particular, \(\Sigma_n(a) \neq -1\) for any
      \(a \in \Z\).
    \end{Le}

    \begin{proof}
      By Lemma~\ref{lem:mult}, \(\circ_n\)-units map to units in \(\Z\). Hence, we have
      \(\Sigma_n(a) = \pm 1\). If \(\Sigma_n(a) = 1\), then \((n+1)a = 0\) and so \(a =
      0\). On the other hand, if \(\Sigma_n(a) = -1\) then \((n+1)a = 2\). But this has
      no integer solution for \(n \geq 2\).
    \end{proof}

    \begin{Cor} \label{cor:sharp}
      The monoid \((\Z, \circ_n, 0)\) is a sharp (only unit is the identity).
    \end{Cor}

\subsection{Primes and irreducibles in \(\nEl(\Z)\)}

    For \(\circ_n\)-\emph{primes}, the equivalence with classical primes holds in full
    generality, and is the key result linking \(\nEl(\Z)\) to classical number theory.

    \begin{Pro} \label{prop:prime_iff}
      A non-zero non-unit \(p \in \Z\) is \(\circ_n\)-prime if and only if
      \(|\Sigma_n(p)|\) is a classical prime. In this case, \(|\Sigma_n(p)| \equiv \pm 1
      \pmod{n+1}\).
    \end{Pro}

    \begin{proof}
      \((\Leftarrow)\) Lemma~\ref{cor:prime_suff}.

      \((\Rightarrow)\) Let \(|\Sigma_n(p)| = m = d\cdot e\) with \(1 < d, e < m\). By
      Lemma~\ref{lem:coprime}, \(\gcd(m, n+1) = 1\), so by the Chinese Remainder Theorem,
      the arithmetic progression \(1 + (n+1)\Z\) meets every residue class modulo \(m\).
      Choose \(s \equiv d \pmod{m}\) and \(t \equiv e \pmod{m}\) with \(s, t \in 1 +
      (n+1)\Z\). By Lemma~\ref{lem:sigma_image}, there exist \(a, b \in \Z\) with
      \(\Sigma_n(a) = s\) and \(\Sigma_n(b) = t\). Then \(\Sigma_n(p) \mid st =
      \Sigma_n(a)\Sigma_n(b)\), meaning \(p \mid ab\), but \(\Sigma_n(p) \nmid
      \Sigma_n(a)\) and \(\Sigma_n(p) \nmid \Sigma_n(b)\) (since \(d, e < m\)) and thus,
      \(p\nmid a\) and \(p\nmid b\), so \(p\) is not \(\circ_n\)-prime.
    \end{proof}

    \begin{Rem}
      We showed in Proposition \ref{pro:factorisation_irred} that every element can be
      written as a product of irreducibles. The same statement does {\bfseries not} hold
      for \(\circ_n\)-primes in general. We already saw a counterexample in
      Example~\ref{ex:main_counter_example} for \(\nEl(\Z)\):

      The element \(-1 \in \nEl(\Z)\) for \(n=4\) is \(\circ_4\)-irreducible but not
      \(\circ_4\)-prime, and has no factorisation into \(\circ_4\)-primes.
    \end{Rem}

    \begin{Pro}[Uses Dirichlet] \label{pro:irred=prime}
      The \(\circ_n\)-irreducible elements agree with the \(\circ_n\)-primes in
      \(\nEl(\Z)\) if and only if \((\Z/(n+1)\Z)^{\times} \cong \{\pm 1\}\).
    \end{Pro}

    \begin{proof}
      We need to show that irreducibles are prime (since the reverse is always true). For
      this, it suffices to show that their norms are prime, thanks to
      Propositions~\ref{prop:prime_iff}. Meaning me must show that every irreducible has
      prime norm if and only if \((\Z/(n+1)\Z)^{\times} \cong \{\pm 1\}\). The norm
      reduces this question solely to the submonoid \(M_n := (1+(n+1)\Z,\cdot,1)\) via
      Lemma~\ref{lem:sigma_image}.

      \noindent {\itshape Necessity:} Suppose there exists \(1\neq u \in
      (\Z/(n+1)\Z)^{\times}\). By Dirichlet, there exist classical primes \(p\equiv
      u\pmod{n+1}\) and \(p'\equiv u^{-1}\pmod{n+1}\). Let \(q\) be their product. It
      lies in the image of \(\Sigma_n\) as \(q \equiv u\cdot u^{-1} \pmod{n+1} = 1\pmod
      {n+1}\). Hence, there exists an element \(a\in \nEl(\Z)\) such that \(\Sigma_n(a) =
      q\). By construction, \(a\) can not be \(\circ_n\)-prime. But it is clearly
      \(\circ_n\)-irreducible since its norm is the product of primes, which are not
      realisable as norms of elements of \(\nEl(\Z)\).\\

      \noindent {\itshape Sufficiency:} Let \((\Z/(n+1)\Z)^{\times} = \{\pm 1\}\) and \(q
      \equiv 1\pmod{n+1}\) be the norm of an \(\circ_n\)-irreducible element \(a \in
      \nEl(\Z)\). We have to show that \(|q|\) is a classical prime. Assume \(|q| = ab\)
      with \(1 < a, b < |q|\). Since \(|q| = ab\) is coprime to \(n+1\)
      (Lemma~\ref{lem:coprime}), the factors \(a\) and \(b\) must likewise be coprime.
      Moreover, \(|a| \equiv |b| \equiv 1\pmod{n+1}\) by our assumption on the unit
      group. Trivial arguments about the sign of \(a\) and \(b\) and the fact that the
      norm map \(\Sigma_n\) bijects \(\Z\) onto \(1 + (n+1)\Z\) now implies that \(a, b\)
      or \(-a, -b\) lift to elements in \(\nEl(\Z)\) and will thus give us a
      factorisation of \(q\), contradicting our assumption.
    \end{proof}

\section{Infinitely many \(\circ_n\)-primes and primes in arithmetic
  progressions}\label{sec:dirichlet}

  We now turn to the main parts of this paper, the links of \(n\)-ary ring theory to
  classic arithmetic and, in particular, giving a proof that the existence of
  infinitely many \(\circ_n\)-primes (Euclid's theorem for \(n\)-ary primes) is
  equivalent to Dirichlet's theorem on arithmetic progressions for the special case of
  \(an + 1\). This link is the main reason I believe a deeper exploration of \(n\)-ary
  elliptic algebra might be warranted.

\subsection{Infinitely many \(\circ_n\)-irreducibles}

    While I was not able to prove an equivalent theorem for \(\circ_n\)-primes without
    invoking Dirichlet, below is an Euclid-style proof for \(\circ_n\)-irreducibles.

    \begin{Th}[Euclid's theorem for \(\circ_n\)-irreducibles] \label{thm:euclid_irred}
      Any \(n\neq 0\) has a factorisation in \(\circ_n\)-irreducibles. In particular,
      there are infinitely many \(\circ_n\)-irreducibles in \(\nEl(\Z)\).
    \end{Th}

    \begin{proof}
      \textit{Factorisation in irreducibles.} Every non-zero \(a_0\in \Z\) is either
      \(\circ_n\)-irreducible or \(a_0 = a_1 \circ_n b_0\) with \(a_1, b_0 \neq 0\). To
      avoid confusion, we call again that \(0\in\Z\) corresponds to the sole unit of
      \(\nEl(\Z)\) (Lemma \ref{lem:units}) and as \(n+1\) is not invertible in
      \(\nEl(\Z)\), we have no absorbing (\(\circ_n\)-zero) element.

      Since the monoid is sharp by Corollary~\ref{cor:sharp}, we have \(|\Sigma_n(b_0)| >
      1\) and \(|\Sigma_n(a_1)| < |\Sigma_n(a_0)|\). This descent must terminate,
      yielding a factorisation into \(\circ_n\)-irreducibles. \\

      \noindent\textit{Infinitely many.} Suppose \(\{q_1, \ldots, q_r\}\) is a complete
      finite list of irreducibles. Set
      \[ N \;:=\; q_1 \circ_n \cdots \circ_n q_r \,\qtext{and}\, M \;:=\; 1 - N. \]

      \medskip\noindent\textbf{Case 1: \(M = 0\).}\enspace It follows \(N = 1\), so
      \(\Sigma_n(N) = \Sigma_n(1) = -n\). Consider
      \[ \Sigma_n(2) \;=\; 1 - 2(n+1) \;=\; -(2n+1), \]
      which gives us \(|\Sigma_n(2)| = 2n+1 > 1\). Subsequently, \(2\) is not a unit and
      must have an \(\circ_n\)-irreducible factor \(p\) with \(\Sigma_n(p) \mid -(2n+1)\)
      (this includes the case \(2 = p\)). As \(p \in \{q_1,\ldots,q_r\}\) by assumption,
      \(\Sigma_n(p)\) must divide \(\Sigma_n(N) = -n\) and hence, must divide \(\gcd(n,
      2n+1) = 1\). This forces \(|\Sigma_n(p)| = 1\), which is true exactly for \(p =
      0\). But this is a contradiction.

      \medskip\noindent\textbf{Case 2: \(M \neq 0\).}\enspace Then \(|\Sigma_n(M)| \geq
      2\) (since \(M \neq 0\) and the monoid is sharp), so \(M\) has an
      \(\circ_n\)-irreducible factor \(p\) with \(\Sigma_n(p) \mid \Sigma_n(M)\). By
      assumption, \(p = q_i\) for some \(i\). As such, \(\Sigma_n(q_i)\) divides both
      \(\Sigma_n(N)\) and \(\Sigma_n(M)\), and hence their sum (Lemma~\ref{lem:sum}):
      \[ \Sigma_n(N) + \Sigma_n(M) \;=\; \Sigma_n(N) + \Sigma_n(1-N) \;=\; 1-n. \]
      But \(|\Sigma_n(q_i)| \geq 2\) and the smallest value in \(1+(n+1)\Z\) with
      absolute value \(> 1\) is \(-n\) (with \(\Sigma_n(1) = -n\)). Since \(|-n| = n >
      n-1 = |1-n|\), \(\Sigma_n(q_i)\) cannot divide \(1-n\). We arrive at a
      contradiction.
    \end{proof}

\subsection{Infinitely many \(\circ_n\)-primes and primes in arithmetic
    progressions}

    \begin{Th}[Main theorem] \label{thm:main}
      The \(\circ_n\)-primes are exactly the classical primes \(p\in \Z\) with
      \(|\Sigma_n(p)| \equiv \pm 1 \pmod{n+1}\). Moreover, the following are equivalent:
      \begin{enumerate}
        \item[(i)] There are infinitely many \(\circ_n\)-prime elements in \(\nEl(\Z)\).
        \item[(ii)] There are infinitely many classical primes \(p \equiv \pm 1
          \pmod{n+1}\).
      \end{enumerate}
      Part (ii) is, of course, Dirichlet's theorem for \(\pm1 + (n+1)\Z\) and
      subsequently, there are infinitely many \(\circ_n\)-primes in \(\nEl(\Z)\).
    \end{Th}

    \begin{proof}
      That \(|\Sigma_n(p)| \equiv \pm 1 \pmod{n+1}\) must hold follows directly from
      Proposition~\ref{prop:prime_iff} and Lemma~\ref{lem:coprime}.

      The equivalence of (i) and (ii) follows from the bijection \(\Sigma_n: \Z \to
      1+(n+1)\Z\) (Lemma~\ref{lem:sigma_image}). The \(\circ_n\)-primes in \(\nEl(\Z)\)
      are thus exactly the classical primes \(q\equiv \pm 1\pmod{n+1}\).
    \end{proof}

\section{The \(n\)-ary elliptic class group of \(\nEl(\Z)\)} \label{sec:class_group}

  In light of the Theorem~\ref{thm:main}, we consider the arithmetic study of \(n\)-ary
  ring theory to be of some interest. It is not merely about proving Dirichlet's
  theorem in a purely algebraic manner (though that would, of course, be a highlight
  result), but that it seems to be the more natural setting for certain types of
  arithmetic questions.

  One of the most classical objects in algebraic number theory is of course the class
  group. It can be defined in numerous ways (see Picard group), but the arguably the
  most direct definition is as a quotient of fractional ideals by principal ones. The
  order of this quotient group is known as the class number and it measures how far a
  number ring is form being an unique factorisation domain (UFD). It is a pleasant and
  fairly obvious fact that the group of fractional ideals is isomorphic to the
  Grothendieck group of the monoid of ideals under multiplication.

  Using this definition, we introduce the class group of the standard \(n\)-ary ring
  \(\nEl(\Z)\) associated to \(\Z\) in this section and show that it likewise measures
  unique factorisation in \(\nEl(\Z)\), see Section~\ref{sec:cl_ufd}. This, in turn,
  relates back to Dirichlet's theorem for progressions of the form \(an + 1\).

\subsection{The monoid of \(n\)-ary elliptic ideals and its group completion}

    In Section~\ref{sec:ideals_nEl}, we constructed \(J(\fa)\) as the pre-image of the
    norm function \(\Sigma_n\) and showed that the non-empty \(n\)-ary elliptic ideals
    of \(\nEl(\Z)\) correspond exactly to ideal of \(\Z\) whose generator \(m\) is
    coprime to \(n+1\). We called this the norm-ideal correspondence. Note, here we
    used the basic fact that \(\Z\) is a PID and to clarify notation, we will write
    \[ J(m) \;:=\; J\bigl((m)\bigr) \;=\; \{a\in\Z: (n+1)a\equiv 1\pmod{m}\}. \]

    Using this correspondence, we introduce
    \[ \cN_n \;:=\; \{m\in\N_{>0}: \gcd(m,n+1)=1\}, \]
    as we have just argued that these are exactly the numbers (ideals) that correspond
    to non-empty \(n\)-ary ideals in \(\nEl(\Z)\). Endowing it with a monoid structure
    by (standard) multiplication, it is obvious that the assignment \(m\to J(m)\) is a
    monoid isomorphism. We identify our \(n\)-ary ideal monoid with \(\cN_n\) and form
    its Grothendieck group
    \[
      \cF_n \;:=\; \bigl\{r\in\Q_{>0} : v_p(r)
      =0 \,\text{ for all primes }\, p\mid n+1\bigr\}.
    \]
    As mentioned in our preliminary talk about the class group, \(\cF_n\) becomes the
    group of fractional ideals. We can describe \(\cF_n\) readily as the free abelian
    group on the set of (classical) primes not dividing \(n+1\) and write it as
    \[ \cF_n \;=\; \bigoplus_{p\nmid n+1}\Z\,[p]. \]
    In this notation, the ideal corresponding to \(m\in\cN_n\) becomes \(\sum_{p\mid
    m}v_p(m)\,[p]\).

    Under this correspondence, principal \(n\)-ary ideals of \(\nEl(\Z)\) are of the
    form
    \[ (a) \;=\; J\bigl((\Sigma_n(a))\bigr), \quad a\in \Z. \]
    We can see that as \(a\) ranges over \(\Z\), the values \(|\Sigma_n(a)|\) range
    over
    \[
      \cP_n \;:=\; \bigl\{|1-(n+1)a| : a\in\Z\bigr\}
      \;=\; \bigl\{m\in\N_{>0} : m\equiv\pm 1\pmod{n+1}\bigr\}.
    \]
    Based on this, we define the following subgroup of \(\cF_n\), to be the free group
    \(\overline{\cP}_n\) generated by the set the set
    \[ \{ [p] \mid \cP_n \ni p \nmid n+1 \,\text{ is a classical prime}\, \}. \]
    Note that \([p]\) is just notation to distinguish it from the associated number,
    and in a certain way, is something we regard as a divisor.

    This brings us to the following definition:

\subsection{The class group and its connection to the unit groups}

    \begin{De}
      Define the \emph{\(n\)-ary elliptic class group} of \(\nEl(\Z)\) as
      \[ \Cl_n(\Z) \;:=\; \cF_n \,\big/\, \overline{\cP}_n. \]
    \end{De}

    Under this equivalence, two \(n\)-ary ideals \(J(m)\) and \(J(k)\) represent the
    same class if and only if
    \[ mk^{-1} \;=\; \prod_i p_i^{e_i}, \]
    where each \(p_i\) \(\equiv\pm 1\pmod{n+1}\) is a classical prime.

    Unlike the classical class group, the \(n\)-ary class group is not trivial for
    \(\Z\) (that is to say, its not trivial for \(\nEl(\Z)\)). However, it is fairly
    easy to calculate in this case. We have:

    \begin{Pro} \label{prop:class_group_computation}
      There is a natural isomorphism of groups
      \[ \Cl_n(\Z) \cong (\Z/(n+1)\Z)^{\times} \big/ \{\pm 1\}. \]
      In particular, \(|\Cl_n(\Z)| = \phi(n+1)/2\).
    \end{Pro}

    \begin{proof}
      We define the group homomorphism \(\psi: \cF_n \to (\Z/(n+1)\Z)^{\times}/\{\pm
      1\}\) on generators by
      \[ \psi\bigl([p]\bigr) \;:=\; [p \bmod{n+1}]. \]
      Recall that generators are exactly the classical primes \(p\nmid n+1\) and are thus
      units mod \(n+1\), which gives us well-definidness.

      Surjectivity is just Dirichlet's theorem: For any
      \([u]\in(\Z/(n+1)\Z)^{\times}/\{\pm 1\}\), we can find a prime \(p\equiv
      u\pmod{n+1}\) with \(p\nmid n+1\), meaning a prime \(p\) with \(\psi([p]) = [u]\).

      To understand the kernel, we note that a generator \([p]\) lies in \(\ker\psi\) if
      and only if \(p\equiv\pm 1 \pmod{n+1}\). But this gives us exactly that the kernel
      is \(\overline{\cP}_n\) and thus the theorem.
    \end{proof}

    In particular, we have:

    \begin{Cor}
      \(\Cl_n(\Z)\) is trivial if and only if \((\Z/(n+1)\Z)^{\times} = \{\pm 1\}\),
      meaning precisely for \(n\in\{2,3,5\}\).
    \end{Cor}

    \begin{Cor}
      The \(n\)-ary elliptic class group \(\Cl_n(\Z)\) is finite of order
      \(\phi(n+1)/2\).
    \end{Cor}

\subsection{The class group governs unique factorisation in prime ideals}
    \label{sec:cl_ufd}

    We have the following theorem which sums up our talk on the class group:

    \begin{Th}
      An \(n\)-ary ideal \(J(m)\) of \(\nEl(\Z)\) is principal if and only if its class
      \([J(m)]\in\Cl_n(\Z)\) is trivial. Moreover:
      \begin{enumerate}
        \item[(i)] Every \(n\)-ary ideal \(I\subseteq \nEl(\Z)\) may be uniquely written as
          \[ I \;=\; J(p_1)^{e_1}\cdots J(p_k)^{e_k} \]
          with \(p_i\) classical primes \(p_i\nmid n+1\).
        \item[(ii)] For a \(\circ_n\)-prime \(q\), we have
          \[ (q) \;=\; J(|\Sigma_n(q)|). \]
          In particular, for every non-empty \(n\)-ary prime \(q\) ideal is in the image of
          \(J(-)\).
        \item[(iii)] An \(n\)-ary ideal \(J(p)\) is principal if and only if \(p\equiv\pm 1
          \pmod{n+1}\). Otherwise it represents a non-trivial element of \(\Cl_n(\Z)\).
      \end{enumerate}
      In particular, factorisation into \(n\)-ary prime elements exists exactly when
      \(\Cl_n(\Z)\) is trivial, meaning \(n\in\{2,3,5\}\).
    \end{Th}

    \begin{proof}
      (i): We know from Proposition~\ref{prop:J_ideal_obs} and the norm-ideal
      correspondence that non-empty \(n\)-ary ideals of \(\nEl(\Z)\) correspond exactly
      to the numbers coprime with \(n+1\). Subsequently, we can write \(I = J(m)\) for
      some \(m\) with \(\gcd(m, n+1) = 1\). Factorise \(m = p_1^{e_1} \cdots p_k^{e_k}\)
      and use the fact that \(J\) respects the product (as mentioned in the introduction
      of this section as well) to obtain the desired result.

      (ii): This is just an amalgamion of previous results.
      Proposition~\ref{prop:prime_iff}, says that a non-zero non-unit \(q\in\Z\) is
      \(\circ_n\)-prime if and only if \(|\Sigma_n(q)|\) is a classical prime.
      Lemma~\ref{lem:coprime} states \(|\Sigma_n(q)|\equiv\pm 1\pmod{n+1}\).
      Subsequently, \((q) = J(|\Sigma_n(q)|)\) is a prime \(n\)-ary ideal by
      Proposition~\ref{prop:J_prime}.

      (iii): By the norm-ideal correspondence we know that \(J(p)\) is principal if and
      only if there exists an element \(a\in\Z\) with \(|\Sigma_n(a)| = p\). But this
      just means exactly \(p\equiv\pm 1\pmod{n+1}\). \\

      We return to the initial statement: By part (i), we can write \(J(m) = \prod
      J(p_i)^{e_i}\) and subsequently, its class as
      \[ [J(m)] \;=\; \sum e_i[J(p_i)] \in \Cl_n(\Z). \]
      We know Proposition \ref{prop:class_group_computation} and the definition of
      \(\overline{cP}_n\) that \([J(p_i)]\) is trivial in the class group if and only if
      \(p_i\) is \(\pm1 \pmod{n+1}\). Since different primes can not cancel each other
      out in the class group (as its a subgroup of a free group, hence free), their sum
      is trivial if and only if each summand is trivial and hence, if and only if \(m =
      \pm1 \pmod{n+1}\), meaning exactly when \(m\in \cP_n\) and thus the claim.

      The initial statement is now immediate: Every non-empty \(n\)-ary ideal factors
      into prime ideals of the form \(J(p)\). Each \(J(p)\) is principal if and only if
      \(p\equiv\pm 1\). Moreover, \(\mathrm{Cl}_n(\Z)\) is trivial if and only if all
      primes \(p\nmid n+1\) satisfy \(p\equiv\pm 1\), which itself holds if and only if
      \((\Z/(n+1)\Z)^{\times} = \{\pm 1\}\), that is to say exactly when
      \(n\in\{2,3,5\}\).

      The last assertion follows from Proposition~\ref{pro:irred=prime}.
    \end{proof}

    \begin{Rem}
      Part (i) is, of course, a type of \(n\)-ary Dedekind's factorisation theorem for
      \(\nEl(\Z)\).
    \end{Rem}

\section{The case when \(n+1\) is invertible in \(R\)}

  We will not go back to a more general discussion, as throughout, a special case
  stands out, namely when \(n+1 \in R^\times\) is invertible in our ring \(R\). In that
  case, our theory becomes significantly simpler and we are able to define kernels and
  the isomorphism theorem and certain maps become isomorphisms, simplifying our
  discussions. The general idea is that, other than intrinsic interest, might also be
  used for when \(n+1\) is a non-zero divisor in \(R\). In that case, we can consider
  \(R_{n+1}\) and hope that it might give us useful information. We demonstrate this
  general idea quickly in Subsection~\ref{sec:localisation}, even if do not highlight
  any special results we obtain in this manner.\\

  \noindent {\bfseries New convention:} From now on, unless explicitly stated
  otherwise, \(n+1\) will be assumed to be invertible throughout in every base ring.
  Moreover, \(z = z_R: \frac{1}{n+1}\) will be fixed. Note that morphisms respect it by
  default, meaning \(f(z_A) = z_B\), and ideals contain it (that is to say, ideals are
  assumed to be non-empty).

\subsection{Ideals and Kernels}

    In ordinary ring theory (or group theory), the kernel is one of the most
    fundamental constructions, as it controls the fibres of a homomorphism and allows
    us to consider the first isomorphism theorem, among many other things. It is
    defined as the preimage of the additive unit, which is also the absorbing element
    of the multiplication. However, our "addition", \(\ast\), does not have a unit and
    neither does \(\circ\) have an absorbing element, in general.

    If, however, our \(n\)-ary elliptic ring \(A\) is of the form \(A\simeq \nEl(R)\)
    for ring \(R\) with \(n+1\in R^{\times}\), we can consider \(f^{-1}(\{z\})\), where
    \(z = 1/(n+1)\).

    As mentioned in Remark~\ref{rem:n+1_invertible}, the case when \(n+1\) is
    invertible is essentially its own category and we will do a small exploration of
    this topic here.

\subsubsection*{Kernel definition and ideal property}

      \begin{De}
        For \(f: A \to B\) a morphism of \(n\)-ary elliptic rings, define the
        \emph{\(n\)-ary kernel} as
        \[ \ker_n(f) \;=\; f^{-1}(\{z_B\}). \]
      \end{De}

      \begin{Le} \label{lem:ker_ideal}
        The Kernel \(\ker_n(f)\) of an \(n\)-ary ring morphism \(f: A \to B\) is an
        \(n\)-ary ideal of \(A\).
      \end{Le}

      \begin{proof}
        This is de facto Lemma~\ref{le:zero}. The closure under \(\ast_n\) (I1) follows
        from Lemma~\ref{le:zero}(ii) as \(f\) preserves \(\ast_n\) and we thus have
        \[
          f(x_1 \ast_n \cdots \ast_n x_n)
          \;=\; f(x_1) \ast_n \cdots \ast_n f(x_n)
          \;=\; \underbrace{z \ast_n \cdots \ast_n z}_{ \,\text{n}\, } \;=\; z,
        \]
        showing that \(x_1 \ast_n \cdots \ast_n x_n \in \ker_n(f)\). The absorption under
        \(\circ_n\) (I2) follows from Lemma~\ref{le:zero}(i), which implies
        \[ f(x \circ_n r) \;=\; f(x) \circ_n f(r) \;=\; z \circ_n f(r) \;=\; z, \]
        as \(f\) preserves \(\circ_n\).
      \end{proof}

\subsubsection*{The kernel captures the congruence}

      \begin{Pro} \label{prop:ker_congruence}
        Let \(f: A \to B\) be a morphism of \(n\)-ary elliptic rings. For \(x, y \in A\),
        \[ x \sim_{\ker_n(f)} y \;\iff\; f(x) \;=\; f(y). \]
      \end{Pro}

      \begin{proof}
        \((\Rightarrow)\) Suppose \(x \ast_n a_1 \ast_n \cdots \ast_n a_{n-1} = y \ast_n
        b_1 \ast_n \cdots \ast_n b_{n-1}\) with \(a_i, b_j \in \ker_n(f)\), so \(f(a_i) =
        f(b_j) = z_B\). Applying \(f\) gives us
        \[
          f(x) \ast_n z_B \ast_n \cdots \ast_n z_B
          \;=\;
          f(y) \ast_n z_B \ast_n \cdots \ast_n z_B.
        \]
        We act with \(\ast (z_B, \ldots, z_B)\) on both sides and use axiom (EG2) to
        obtain \(f(x) = f(y)\).

        \((\Leftarrow)\) On the reverse, assume \(f(x) = f(y)\). We define \(\delta := y
        - x \in A\). Set
        \[
          a_1 \;:=\; z_A + \delta,
          \quad a_i \;:=\; z_A\;\; (2 \leq i \leq n-1),
          \quad b_j \;:=\; z_A\;\; (1 \leq j \leq n-1).
        \]
        It follows that
        \[
          \sum_{i=1}^{n-1} a_i - \sum_{j=1}^{n-1} b_j
          \;=\; (z_A + \delta) + (n-2)z_A - (n-1)z_A \;=\; \delta \;=\; y - x.
        \]
        Since \(x \ast_n a_1 \ast_n \cdots \ast_n a_{n-1} = y \ast_n b_1 \ast_n \cdots
        \ast_n b_{n-1}\) is equivalent to \(\sum a_i - \sum b_j = y - x\) (expanding the
        \(\ast_n\)-formula \(1 - \sum(-)\)), we have \(x \sim_{\ker_n(f)} y\). It remains
        to verify \(a_1 \in \ker_n(f)\). We have
        \[
          f(a_1) \;=\; f(z_A + \delta) \;=\; f(z_A) + f(y) - f(x)
          \;=\; z_B + 0 \;=\; z_B,
        \]
        where we used \(f(x) = f(y)\) and \(f(z_A) = z_B\).
      \end{proof}

      \begin{Rem}
        We need \(n+1\in A^{\times}\) in the reverse direction to produce the
        representative \(z_A \in A\).
      \end{Rem}

\subsubsection*{The first isomorphism theorem}

      \begin{Th}[First isomorphism theorem for \(n\)-ary elliptic rings]
        \label{thm:first_iso} Let \(f: A \to B\) be a morphism of \(n\)-ary elliptic
        rings. The following hold:
        \begin{enumerate}
          \item[(i)] \(\ker_n(f)\) is a non-empty \(n\)-ary ideal of \(A\), with \(z_A \in
            \ker_n(f)\).
          \item[(ii)] \(f\) factors uniquely through the quotient. That is to say, there
            is a unique morphism
            \[ \bar{f}: A/\ker_n(f) \to B \]
            with \(\bar{f}([x]) = f(x)\).
          \item[(iii)] The morphism \(\bar{f}\) is injective, and \(f\) induces an
            isomorphism of \(n\)-ary elliptic rings
            \[ A \big/ \ker_n(f) \xrightarrow{\sim} \im(f). \]
        \end{enumerate}
      \end{Th}

      \begin{proof}
        This is Lemma~\ref{lem:ker_ideal} for part (i) and
        Proposition~\ref{prop:ker_congruence} for parts (ii) and (iii).
      \end{proof}

\subsubsection*{The kernel and the \(J\)-construction}

      When \(n+1 \in R^{\times}\), the \(J\)-construction gives a complete description
      of quotients.

      \begin{Pro} \label{prop:kernel_is_J}
        Let \(g:R \to S\) be a ring homomorphism. Then
        \[ \ker_n\bigl(\nEl(g)\bigr) \;=\; J\bigl(\ker(g)\bigr), \]
        where \(\ker(g) = g^{-1}(0) \subseteq R\) is the classical kernel.
      \end{Pro}

      \begin{proof}
        Take \(x \in R\). We have
        \begin{align*}
          x \in \ker_n(\nEl(g)) & \iff g(x) = z_S = \tfrac{1}{n+1} \\[2pt]
                                & \iff (n+1)g(x) = 1               \\[2pt]
                                & \iff g\big((n+1)x) - 1 = 0       \\[2pt]
                                & \iff g\bigl(1 - (n+1)x\bigr) = 0 \\[2pt]
                                & \iff \Sigma_n(x) \in \ker(g)     \\[2pt]
                                & \iff x \in J\!\bigl(\ker(g)\bigr). \qedhere
        \end{align*}
      \end{proof}

      \begin{Th} \label{thm:quotient_iso}
        Let \(\fa \subseteq R\) be any ideal. The canonical ring homomorphism \(\pi: R
        \to R/\fa\) induces an isomorphism of \(n\)-ary elliptic rings
        \[ \nEl(R) \big/ J(\fa) \cong \nEl(R/\fa). \]
      \end{Th}

      \begin{proof}
        The quotient map \(\pi: R \twoheadrightarrow R/\fa\) induces a surjective
        morphism \(\nEl(\pi): \nEl(R) \twoheadrightarrow \nEl(R/\fa)\). By Proposition
        \ref{prop:kernel_is_J},
        \[ \ker_n(\nEl(\pi)) \;=\; J(\ker(\pi)) \;=\; J(\fa). \]
        Theorem \ref{thm:first_iso} now gives \(\nEl(R)/J(\fa) \cong \im(\nEl(\pi)) =
        \nEl(R/\fa)\).
      \end{proof}

      The following diagram summarises the interplay between classical kernels, the
      \(J\)-construction, and \(n\)-ary kernels. It commutes by
      Proposition~\ref{prop:kernel_is_J}.
      \[
        \begin{array}{ccccccc}
          \ker(g)                       &       & \subseteq       &           & R                                & \xrightarrow{g}       & S                                \\[6pt]
          \Big\downarrow\scriptstyle{J} &       &                 &           & \Big\downarrow\scriptstyle{\nEl} &                       & \Big\downarrow\scriptstyle{\nEl} \\[4pt]
          J(\ker g)                     & \;=\; & \ker_n(\nEl(g)) & \subseteq & \nEl(R)                          & \xrightarrow{\nEl(g)} & \nEl(S)
        \end{array}
      \]

\subsubsection{Cancellativity and \(n\)-nary fields: Ideal definitions}

      Recall Definitions~\ref{def:cancellative} and \ref{def:nary_field}. These were
      defined using the properties we want ``on the face'', not the ``algebraically
      sound way''. Of course, can always define these terms using prime and maximal
      ideals, but we need \(n+1\) to be invertible (we need \(1/(n+1)\)) to work these
      constructions sensibly.

      \begin{De} \label{def:ideal_cancel}
        An \(n\)-ary elliptic ring \(A\) is \emph{ideal-cancellative} if there exist an
        \(n\)-ary elliptic ring \(A'\) and an \(n\)-ary prime ideal \(\p\subseteq A'\)
        such that
        \[ A \cong A'\big/\p. \]
      \end{De}

      \begin{De} \label{def:ideal_field}
        An \(n\)-ary elliptic ring \(A\) is an \emph{ideal-field} if there exist an
        \(n\)-ary elliptic ring \(A'\) and an \(n\)-ary maximal ideal \(\fm\subseteq A'\)
        such that
        \[ A \cong A'\big/\fm. \]
      \end{De}

      \begin{Note} \label{nt:cancellative_field}
        To work sensibly in our category where \(n+1\) is invertible, we need to amend
        the two definitions and say that \(\nEl(R)\) is ideal-cancellative (respectively
        an ideal-field) if there exists a ring \(S\) with \(n+1\in S^{\times}\), such
        that \(\nEl(R)\simeq \nEl(S)/\p\) where \(\p\subseteq(\nEl(S))\) is an
        \(\circ_n\)-prime (respectively \(\circ_n\)-maximal) ideal.
      \end{Note}

      We will need the following observation throughout:

      \begin{Le} \label{lem:norm_iso}
        The norm map \(\Sigma_n: \nEl(R) \xto{\sim} R\) is a monoid isomorphism.
      \end{Le}

      \begin{proof}
        Define the inverse of \(p\mapsto 1-(n+1)p\),\; as\; \(q\mapsto(1-q)/(n+1)\).
      \end{proof}

      \begin{Cor} \label{cor:cancel_invertible}
        \(\nEl(R)\) is \(\circ_n\)-cancellative if and only if \(R\) is an integral
        domain.
      \end{Cor}

      \begin{proof}
        Lemma \ref{lem:norm_iso} tells us that the norm map is an isomorphism. By
        Proposition~\ref{prop:cancel_char}, \(\nEl(R)\) is \(\circ_n\)-cancellative if
        and only if every nonzero element of \(R\) is a non-zero-divisor, but this just
        means that \(R\) is an integral domain.
      \end{proof}

      \begin{Rem}
        When \(n+1 \notin R^{\times}\), \(\circ_n\)-cancellativity is strictly weaker
        than \(R\) being an integral domain. We can only deduce that elements in the
        proper subset \(\im(\Sigma_n) = 1+(n+1)R \subsetneq R\) are non-zero-divisors.
        For example, take \(R = \Z/9\Z\) and \(n = 2\) (so \(n+1 = 3 \notin
        R^{\times}\)). Then \(\im(\Sigma_2) = 1 + 3R = \{1, 4, 7\} \pmod{9}\). Each of
        these is a unit in \(\Z/9\Z\), so every element of
        \(\im(\Sigma_2)\setminus\{0\}\) is a non-zero-divisor. Hence \(\nEl(\Z/9\Z)\) is
        \(\circ_2\)-cancellative.
      \end{Rem}

      \begin{Pro}
        \(\circ_n\)-cancellativeness (Definition~\ref{def:cancellative}) is equivalent to
        ideal-cancellativeness.
      \end{Pro}

      \begin{proof}
        Let \(\nEl(R)\) be \(\circ_n\)-cancellative. We know form
        Corollary~\ref{cor:cancel_invertible} that \(R\) is an integral domain, and
        subsequently \((0)\) is a prime ideal of \(R\). Hence, \(J((0))\) is an \(n\)-ary
        prime ideal of \(\nEl(R)\), and \(\nEl(R)\cong\nEl(R)/J((0))\). Thus \(\nEl(R)\)
        is ideal-cancellative with \(A'=\nEl(R)\) and \(\p=J((0))\).

        We point to Note~\ref{nt:cancellative_field}. The norm-ideal correspondence
        (Theorem~\ref{thm:quotient_iso} and Proposition~\ref{prop:J_prime}) identifies
        \(\p = J(\q)\) with a classical prime ideal \(\q\subseteq S\), and
        \(\nEl(S)/J(\q)\cong\nEl(S/\q)\). Since \(S/\q\) is an integral domain, and
        \(\nEl(R)\cong\nEl(S/\q)\) implies \(R\cong S/\q\) as rings
        (Lemma~\ref{lem:norm_iso}), \(R\) is an integral domain.
        Corollary~\ref{cor:cancel_invertible} allows us to deduce that \(\nEl(R)\) is
        \(\circ_n\)-cancellative.
      \end{proof}

\subsubsection*{The \(\circ_n\)-inverse and \(n\)-ary fields}

      We have seen that \(0_R\) is the identity of the \(\circ_n\)-monoid and that
      \(z\) is the absorbing element. The following lemma computes
      \(\circ_n\)-inverses.

      \begin{Le} \label{lem:circ_inverse}
        An element \(a\in\nEl(R)\) has an \(\circ_n\)-inverse \(b\) if and only if
        \(\Sigma_n(a)\in R^{\times}\), in which case
        \[
          b \;=\; \frac{-a}{\Sigma_n(a)}
          \;=\; z\cdot\left(1-\frac{1}{\Sigma_n(a)}\right).
        \]
        In particular, the absorbing element \(z = 1/(n+1)\) has no \(\circ_n\)-inverse
        (since \(\Sigma_n(z) = 0\)).
      \end{Le}

      \begin{proof}
        This follows from Lemma~\ref{lem:norm_iso}, since an isomorphism of monoids
        remains one if restricted to invertible elements. Simple transformation shows
        that
        \[
          b \;=\; -a/\Sigma_n(a) \;=\; \frac{-(1-\Sigma_n(a))}{(n+1)\Sigma_n(a)}
          \;=\; \frac{1}{n+1} - \frac{1}{(n+1)\Sigma_n(a)},
        \]
        giving us our desired form.
      \end{proof}

      \begin{Th} \label{thm:nary_field_char}
        The following are equivalent:
        \begin{enumerate}
          \item[(i)] \(\nEl(R)\) is an \(n\)-ary elliptic field.
          \item[(ii)] \(R\) is a (classical) field.
          \item[(iii)] \(\nEl(R) \cong \nEl(R/\fm)\) for the (unique) maximal ideal \(\fm
            = (0)\) of \(R\).
          \item[(iv)] \(J((0))=\{z\}\) and every \(n\)-ary ideal of \(\nEl(R)\) is either
            \(\emptyset\), \(\{z\}\), or \(\nEl(R)\).
        \end{enumerate}
      \end{Th}

      \begin{proof}
        \((i)\Leftrightarrow (ii)\): This is Lemmas~\ref{lem:circ_inverse} and
        ~\ref{lem:norm_iso}.

        \((ii)\Leftrightarrow (iii)\): This is immediate since \(R\cong R/(0)\) and
        \((0)\) is the unique maximal ideal of a field.

        \((ii)\Leftrightarrow (iv)\): Recall that \(J(\fa) = \{x\in R\mid
        \Sigma_n(x)\in\fa\}\). By Theorem~\ref{thm:quotient_iso} and
        Proposition~\ref{prop:J_ideal} (the norm-ideal correspondence), we know that the
        \(n\)-ary ideals of \(\nEl(R)\) are exactly the sets \(J(\fa)\), for ideals
        \(\fa\subseteq R\). These are trivial (\(\emptyset\), \(\{z\}\), or \(\nEl(R)\))
        if and only if the ideals of \(R\) are trivial (\((0)\) and \(R\)). But this just
        means \(R\) is a field. The condition \(J((0)) = \{z\}\) says \(\Sigma_n\) has a
        unique zero, which is always true when \(n+1\in R^{\times}\) (namely \(z\)), so
        this condition is automatic.
      \end{proof}

      In like-manner to Theorem~\ref{th:field_is_cancellative}, \(\nEl(R)\) begin an
      ideal-field implies it being ideal-cancellativ. This follows immediately from the
      theorem below:

      \begin{Th}[Maximal implies prime] \label{thm:max_implies_prime}
        A maximal \(n\)-ary ideal \(\fm \subseteq \nEl(R)\) is prime.
      \end{Th}

      \begin{proof}
        By the norm-ideal correspondence, every \(n\)-ary ideal of \(\nEl(R)\) is of the
        from \(J(\fb)\) for a unique ideal \(\fb\subseteq R\). Moreover \(J(\fb)\subseteq
        J(\fb')\) if and only if \(\fb\subseteq\fb'\). The maximality of \(J(\fm)\)
        therefore corresponds exactly to maximality of \(\fm\). Since every maximal ideal
        of a ring is prime, \(\fm\) is prime and thus, \(J(\fm)\) is prime as an
        \(n\)-ary ideal as well.
      \end{proof}

      \begin{Cor}
        Every ideal-field (Definition~\ref{def:ideal_field}) \(A\) is ideal-cancellative
        (Definition~\ref{def:ideal_cancel}).
      \end{Cor}

      \begin{Pro}
        An \(n\)-ary ring \(\nEl(R)\) is an \(n\)-ary elliptic field
        (Definition~\ref{def:nary_field}) if and only if \(\nEl(R)\) is an ideal field
        (Definition~\ref{def:ideal_field}).
      \end{Pro}

      \begin{proof}
        Let \(\nEl(R)\) be an ideal field. Theorem~\ref{thm:nary_field_char} states that
        \(R\) must be a field and as such, \((0)\) is the unique maximal ideal of \(R\).
        This makes \(J((0))\) a maximal \(n\)-ary ideal of \(\nEl(R)\) by
        Theorem~\ref{thm:max_implies_prime} and thus \(\nEl(R) \cong \nEl(R)/J((0))\). We
        deduce that \(\nEl(R)\) is an ideal-field with \(A' = \nEl(R)\) and \(\fm =
        J((0))\).

        Suppose \(\nEl(R)\) is an ideal-field, meaning, \(\nEl(R)\cong A'/\fm\). By
        assumption (Note~\ref{nt:cancellative_field}), we have \(A' = \nEl(S)\) with
        \(n+1\in S^{\times}\). The \(n\)-ary maximal ideals of \(\nEl(S)\) are \(J(\fm)\)
        for maximal ideals \(\fm\subseteq S\), and \(\nEl(S)/J(\fm)\cong\nEl(S/\fm)\)
        with \(S/\fm\) a field. Since \(\nEl(R)\cong\nEl(S/\fm)\) and \(S/\fm\) is a
        field, it follows that \(R\) is a field (as \(n\)-ary ring isomorphism preserves
        the group-like binoid structure). Theorem~\ref{thm:nary_field_char} now gives the
        desired result.
      \end{proof}

\subsection{The localisation \(\nEl(\Z_{n+1})\)} \label{sec:localisation}

    Let \(\Z_{n+1} := S^{-1}\Z\), where \(S = \{(n+1)^m: m \geq 0\}\), be the
    localisation of \(\Z\) with \(n+1\). This is, of course, again a PID. Its prime
    elements are rational primes \(q\) for which \(q\nmid n+1\). The inclusion
    \(\Z\hookrightarrow\Z_{n+1}\) induces an injective \(n\)-ary ring homomorphism
    \(\nEl(\Z)\hookrightarrow\nEl(\Z_{n+1})\) which we wish to study in this short
    section. The key observation is the following proposition:

    As a consequence, we no longer have to worry about being able to lift classical
    primes to \(n\)-ary primes. This removes the main obstruction to unique
    factorisation in \(\circ_n\) primes and the class group becomes always trivial.
    This can be summed up as:

    \begin{Cor}
      The following are equivalent for an element \(p \in \nEl(\Z_{n+1})\):
      \begin{itemize}
        \item[(i)] \(p\) is \(\circ_n\)-irreducible;
        \item[(ii)] \(p\) is \(\circ_n\)-prime
        \item[(iii)] \(\Sigma_n(p)\) is a rational prime not dividing \(n+1\).
      \end{itemize}
      In particular, \(\mathrm{Cl}_n(\Z_{n+1}) = 1\) and every nonzero non-unit factors
      uniquely into \(\circ_n\)-primes. Moreover, there are infinitely many
      \(\circ_n\)-primes.
    \end{Cor}

    The gap between \(\nEl(\Z_{n+1})\), where this is trivial, and \(\nEl(\Z)\), where
    it equals Dirichlet can be captured by the following:
    \begin{quote}
      A \(\circ_n\)-prime \(p\in\nEl(\Z_{n+1})\) lies in \(\Z\subseteq \Z_{n+1}\) if and
      only if \(\Sigma_n(p)\equiv 1\pmod{n+1}\).
    \end{quote}

    We can also phrase it as saying that Diriclets theorem for \(1+(n+1)\Z\) is
    equivalent to the statement that infinitely many \(\circ_n\)-primes of
    \(\nEl(\Z_{n+1})\) are integral. \\

    We sum our discussion up by writing:

    \begin{center}
      \renewcommand{\arraystretch}{1.4}
      \begin{tabular}{lcc}
        \hline
                                           & \(\nEl(\Z)\)                        & \(\nEl(\Z_{n+1})\) \\
        \hline
        \(\Sigma_n\) surjective            & No (\(\mathrm{Im}=1+(n+1)\Z\))      & Yes                \\
        Irred.\ \(=\) Prime                & Only for \(n\in\{2,3,5\}\)          & Always             \\
        Unique factorisation (elements)    & Only for \(n\in\{2,3,5\}\)          & Always             \\
        Class group                        & \((\Z/(n+1)\Z)^{\times}/\{\pm 1\}\) & Trivial            \\
        \(\infty\)-many \(\circ_n\)-primes & \(\Leftrightarrow\) Dirichlet       & Trivial (Euclid)   \\
        Primes correspond to               & \(q\equiv \pm 1\pmod{n+1}\)         & All \(q\nmid n+1\) \\
        \hline
      \end{tabular}
    \end{center}

\end{document}